\documentclass[11pt, letterpaper,reqno]{amsart}

\def\esssup{\mathop{\rm ess sup}} 

\newtheorem{defn}{Definition}[section]
\newtheorem{rem}{Remark}[section]
\newtheorem{thm}{Theorem}[section]
\newtheorem{lem}{Lemma}[section]
\newtheorem{prop}{Proposition}[section]
\newtheorem{coro}{Corollary}[section]

\newtheorem{eg}{Example}

\renewcommand{\P}{\mathbb{P}}

\newcommand{\Q}{\mathbb{Q}}
\newcommand{\R}{\mathbb{R}}
\newcommand{\E}{\mathbb{E}}
\newcommand{\N}{\mathbb{N}}
\newcommand{\F}{\mathcal{F}}
\newcommand{\cH}{\mathcal{H}}

\newcommand{\G}{\mathcal{G}}
\newcommand{\oG}{\overline{\mathcal{G}}}

\newcommand{\cN}{\mathcal{N}}
\newcommand{\Z}{\mathbb{Z}}
\newcommand{\T}{\mathcal{T}}

\newcommand{\bX}{{\bf{X}}}
\newcommand{\bx}{{\bf{x}}}

\newcommand{\A}{\mathcal{A}}
\newcommand{\B}{\mathcal{B}}

\newcommand{\eps}{\varepsilon}

\newcommand{\ms}{\scriptscriptstyle}

\numberwithin{equation}{section}

\title[]{On the Multi-Dimensional Controller-and-Stopper Games}\thanks{We would like to thank Mihai S\^irbu for his thoughtful suggestions. We also would like to thank the two anonymous referees whose suggestions helped us improve our paper.}
\author[]{Erhan Bayraktar}\thanks{This research is supported in part by the National Science Foundation under an applied mathematics research grants DMS-0906257 and  DMS-1118673, and a Career grant, and in part by the Susan M. Smith Professorship.} 
\address[Erhan Bayraktar]{Department of Mathematics, University of Michigan, 530 Church Street, Ann Arbor, MI 48109, USA}
\email{erhan@umich.edu}
\author[]{Yu-Jui Huang}
\address[Yu-Jui Huang]{Department of Mathematics, University of Michigan, 530 Church Street, Ann Arbor, MI 48109, USA}
\email{jayhuang@umich.edu}

\date{January 6, 2013}

\begin{document}

\begin{abstract}
We consider a zero-sum stochastic differential controller-and-stopper game in which the state process is a controlled 
diffusion evolving in a multi-dimensional Euclidean space. In this game, the controller affects both the drift and diffusion terms of the state process, and the diffusion term can be degenerate. Under appropriate conditions, we show that the game has a value and the value function is the unique viscosity solution to an obstacle problem for a Hamilton-Jacobi-Bellman equation. 

\end{abstract}

\maketitle

\section{Introduction}
We consider a zero-sum stochastic differential game of control and stopping under a fixed time horizon $T>0$. There are two players, the ``controller'' and the ``stopper,'' and a state process $X^\alpha$ which can be manipulated by the controller through the selection of the control $\alpha$. Suppose the game starts at time $t\in[0,T]$. While the stopper has the right to choose the duration of this game (in the form of a random time $\tau$), she incurs the running cost $f(s,X^{\alpha}_s,\alpha_s)$ at every moment $t\le s<\tau$, and the terminal cost $g(X^\alpha_\tau)$ at the time the game stops. Given the instantaneous discount rate $c(s,X^{\alpha}_s)$, the stopper would like to minimize her expected discounted cost
\begin{equation}\label{cost}
\E\left[\int_{t}^{\tau}e^{-\int_{t}^{s}c(u,X_{u}^{\alpha})du}f(s,X_{s}^{\alpha},\alpha_s)ds + e^{-\int_{t}^{\tau}c(u,X_{u}^{\alpha})du}g(X_{\tau}^{\alpha})\right]
\end{equation}
over all choices of $\tau$. At the same time, however, the controller plays against her by maximizing \eqref{cost} over all choices of $\alpha$.

Ever since the game of control and stopping was introduced by Maitra \& Sudderth \cite{MS96}, it has been known to be closely related to some common problems in mathematical finance, such as pricing American contingent claims (see e.g. \cite{KK98, KW00, KZ05}) and minimizing the probability of lifetime ruin (see \cite{BY11}). The game itself, however, has not been studied to a great extent except certain particular cases. Karatzas and Sudderth \cite{KS01} study a zero-sum controller-and-stopper game in which the state process $X^\alpha$ is a one-dimensional diffusion along a given interval on $\mathbb{R}$. Under appropriate conditions they prove that this game has a value and describe fairly explicitly a saddle point of optimal choices. It turns out, however, difficult to extend their results to multi-dimensional cases, as their techniques rely heavily on theorems of optimal stopping for one-dimensional diffusions. To deal with zero-sum multi-dimensional games of control and stopping, Karatzas and Zamfirescu \cite{KZ08} develop a martingale approach; also see \cite{baykaryao}, \cite{MR2746174} and \cite{MR2746173}. Again, it is shown that the game has a value, and a saddle point of optimal choices is constructed. However, it is assumed to be that the controller can affect only the drift term of $X^{\alpha}$. 

There is yet another subtle discrepancy between the one-dimensional game in \cite{KS01} and the multi-dimensional game in \cite{KZ08}: the use of ``strategies''. Typically, in a two-player game, the player who acts first would not choose a fixed static action. Instead, she prefers to employ a strategy, which will give different responses to different future actions the other player will take. This additional flexibility enables the player to further decrease (increase) the expected cost, if she is the minimizer (maximizer). For example, in a game with two controllers (see e.g. \cite{EK74-TAMS, EK74-JMAA, FS89, BL08, BMN12}), the controller who acts first employs a strategy, which is a function that takes the other controller's latter decision as input and generates a control. Note that the use of strategies is preserved in the one-dimensional controller-and-stopper game in \cite{KS01}: what the stopper employs is not simply a stopping time, but a strategy in the form of a random time which depends on the controller's decision. This kind of dynamic interaction is missing, however, in the multi-dimensional case: in \cite{KZ08}, the stopper is restricted to use stopping times, which give the same response to any choice the controller makes.

Zero-sum multi-dimensional controller-and-stopper games are also covered in Hamad{\`e}ne \& Lepeltier \cite{HL00} and Hamad\`{e}ne \cite{Hamadene06}, as a special case of mixed games introduced there. The main tool used in these papers is the theory of backward differential equations with two reflecting barriers. Interestingly, even though the method in \cite{HL00, Hamadene06} differs largely from that in \cite{KZ08}, these two papers also require a diffusion coefficient which is not affected by the controller, and do not allow the use of strategies. This is in contrast with the one-dimensional case in \cite{KS01}, where everything works out fine without any of the above restrictions. It is therefore of interest to see whether we can construct a new methodology under which multi-dimensional controller-and-stopper games can be analyzed even when the conditions required in \cite{KZ08,HL00,Hamadene06} fail to hold.

In this paper, such a methodology is built, under a Markovian framework. On the one hand, we allow both the drift and diffusion terms of the state process $X^{\alpha}$ to be controlled. On the other hand, we allow the players to use strategies. Specifically, we first define {\it non-anticipating} strategies in Definition~\ref{defn:strategy}. Then, in contrast to two-controller games where both players use strategies, only the stopper chooses to use strategies in our case (which coincides with the set-up in \cite{KS01}). This is because by the nature of a controller-and-stopper game, the controller cannot benefit from using non-anticipating strategies; see Remark~\ref{rem:no control strategy}. With this observation in mind, we give appropriate definitions of the upper value function $U$ and the lower value function $V$ in \eqref{U} and \eqref{V} respectively. Under this set-up, one presumably could construct a saddle point of optimal choices by imposing suitable assumptions on the cost functions, the dynamics of $X^\alpha$, the associated Hamiltonian, or the control set (as is done in \cite{KS01, KZ08, Hamadene06, HL00}; see Remark~\ref{rem:eps optimal}). However, we have no plan to impose assumptions for constructing a saddle point. Instead, we intend to work under a rather general framework, and determine under what conditions the game has a value (i.e. $U=V$) and how we can derive a PDE characterization for this value when it exists.

Our method is motivated by Bouchard \& Touzi \cite{BT11}, where the weak dynamic programming principle for stochastic control problems was first introduced. By generalizing the weak dynamic programming principle in \cite{BT11} to the context of controller-and-stopper games, we show that $V$ is a viscosity supersolution and $U^*$ is a viscosity subsolution to an obstacle problem for a Hamilton-Jacobi-Bellman equation, where $U^*$ denotes the upper semicontinuous envelope of $U$ defined as in \eqref{envelopes}. More specifically, we first prove a continuity result for an optimal stopping problem embedded in $V$ (Lemma~\ref{lem:conti}), which enables us to follow the arguments in \cite[Theorem 3.5]{BT11} even under the current context of controller-and-stopper games. We obtain, accordingly, a weak dynamic programming principle for $V$ (Proposition~\ref{prop:WDPP_super}), which is the key to proving the supersolution property of $V$ (Propositions~\ref{prop:vis super}). On the other hand, by generalizing the arguments in Chapter 3 of Krylov \cite{Krylov-book-80}, we derive a continuity result for an optimal control problem embedded in $U$ (Lemma~\ref{lem:L continuous}). This leads to a weak dynamic programming principle for $U$ (Proposition~\ref{prop:WDPP_sub}), from which the subsolution property of $U^*$ follows (Proposition~\ref{prop:vis sub}). Finally, under appropriate conditions, we prove a comparison result for the associated obstacle problem. Since $V$ is a viscosity supersolution and $U^*$ is a viscosity subsolution, the comparison result implies $U^*\le V$. Recalling that $U^*$ is actually larger than $V$ by definition, we conclude that $U^*=V$. This in particular implies $U=V$, i.e. the game has a value, and the value function is the unique viscosity solution to the associated obstacle problem. This is the main result of this paper; see Theorem~\ref{prop:U=V}. Note that once we have this PDE characterization, we can compute the value of the game using a stochastic numerical scheme proposed in Bayraktar \& Fahim \cite{BF11}. 

Another important advantage of our method is that it does not require any non-degeneracy condition on the diffusion term of $X^\alpha$. For the multi-dimensional case in \cite{KZ08,HL00,Hamadene06}, Girsanov's theorem plays a crucial role, which entails non-degeneracy of the diffusion term. Even for the one-dimensional case in \cite{KS01}, this non-degeneracy is needed to ensure the existence of the state process (in the weak sense). Note that Weerasinghe \cite{Weerasinghe06} actually follows the one-dimensional model in \cite{KS01} and extends it to the case with degenerate diffusion term; but at the same time, she assumes boundedness of the diffusion term, and some specific conditions including twice differentiability of the drift term and concavity of the cost function.    

It is worth noting that while \cite{KZ08,HL00,Hamadene06} do not allow the use of strategies and require the diffusion coefficient be control-independent and non-degenerate, they allow for non-Markovian dynamics and cost structures, as well as for non-Lipschitz drift coefficients. As a first step to allowing the use of strategies and incorporating controlled, and possibly degenerate, diffusion coefficients in a zero-sum multi-dimensional controller-and-stopper game, this paper focuses on proving the existence and characterization of the value of the game under a Markovian framework with Lipschitz coefficients. We leave the general non-Markovian and non-Lipschitz case for future research.   

The structure of this paper is as follows: in Section 2, we set up the framework of our study. In Section~\ref{sec:problem}, we define strategies, and give appropriate definitions of the upper value function $U$ and the lower value function $V$. In Sections~\ref{sec:super} and \ref{sec:sub}, the supersolution property of $V$ and the subsolution property $U^*$ are derived, respectively. In Section~\ref{sec:comparison}, we prove a comparison theorem, which leads to the existence of the value of the game and the viscosity solution property of the value function.         

\subsection{Notation} We collect some notation and definitions here for readers' convenience.
\begin{itemize}
\item Given a probability space $(E,\mathcal{I},P)$, we denote by $L^0(E,\mathcal{I})$ the set of real-valued random variables on $(E,\mathcal{I})$; for $p\in[1,\infty)$, let $L^p_{n}(E,\mathcal{I},P)$ denote the set of $\R^n$-valued random variables $R$ on $(E,\mathcal{I})$ s.t. $\E_{P}[|R|^p]<\infty$. For the ``$n=1$'' case, we simply write $L^p_1$ as $L^p$.  
\item $\R_+:=[0,\infty)$ and $\mathcal{S}:=\R^d\times\R_+\times\R_+$. 
\item $\mathbb{M}^d$ denotes the set of $d\times d$ real matrices.
\item Given $E\subseteq\R^n$, $\operatorname{LSC}(E)$ denotes the set of lower semicontinuous functions defined on $E$, and $\operatorname{USC}(E)$ denotes the set of upper semicontinuous functions defined on $E$.
\item Let $E$ be a normed space. For any $(t,x)\in[0,T]\times E$, we define two types of balls centered at $(t,x)$ with radius $r>0$ as follows
\begin{equation}\label{defn ball}
\begin{split}
B_r(t,x)&:=\{(t',x')\in[0,T]\times E\mid |t'-t|<r,\ |x'-x|<r\};\\
B(t,x;r)&:=\{(t',x')\in[0,T]\times E\mid t'\in(t-r,t],\ |x'-x|<r\}.
\end{split}
\end{equation}
We denote by $\bar{B}_r(t,x)$ and $\bar{B}(t,x,;r)$ the closures of $B_r(t,x)$ and $B(t,x;r)$, respectively. Moreover, given $w:[0,T]\times E\mapsto\R$, we define the upper and lower semicontinuous envelopes of $w$, respectively, by 
\begin{equation}\label{envelopes}
\begin{split}
w^{*}(t,x)&:=\lim_{\delta\downarrow 0}\sup\{w(t',x')\mid (t',x')\in[0,T)\times E\ \hbox{with}\ (t',x')\in B_\delta(t,x)\};\\
w_{*}(t,x)&:=\lim_{\delta\downarrow 0}\inf\{w(t',x')\mid (t',x')\in[0,T)\times E\ \hbox{with}\ (t',x')\in B_\delta(t,x)\}.
\end{split}
\end{equation}
\end{itemize}


\section{Preliminaries}\label{sec:prelim}
\subsection{The Set-up}\label{subsec:setup}
Fix $T>0$ and $d\in\N$. For any $t\in[0,T]$, let $\Omega^t:=C([t,T];\mathbb{R}^d)$ be the canonical space of continuous paths equipped with the uniform norm $\|\tilde{\omega}\|_{t,T}:=\sup_{s\in[t,T]}|\tilde{\omega}_s|$, $\tilde{\omega}\in\Omega^t$. Let $W^t$ denote the canonical process on $\Omega^t$, and $\mathbb{G}^{t}=\{\G^{t}_s\}_{s\in[t,T]}$ denote the natural filtration generated by $W^t$. Let $\P^t$ be the Wiener measure on $(\Omega^t,\G^{t}_T)$, and consider the collection of $\P^t$-null sets $\cN^t:=\{N\in\G^{t}_T \mid \P^t(N)=0\}$ and its completion $\overline{\cN}^t:=\{A\subseteq\Omega^t \mid A\subseteq N\ \hbox{for some}\ N\in\cN^t\}$. Now, define $\overline{\mathbb{G}}^t=\{\oG^t_s\}_{s\in[t,T]}$ as the augmentation of $\mathbb{G}^{t}$ by the sets in $\overline{\cN}^t$, i.e. $\oG^t_s:=\sigma(\G^{t}_s\cup\overline{\cN}^t)$, $s\in[t,T]$. For any $x\in\R^d$, we also consider $\G^{t,x}_s:=\G^t_s\cap\{W^t_t=x\}$, $\forall s\in[t,T]$. For $\Omega^t$, $W^t$, $\cN^t$, $\overline{\cN}^t$, $\G^t_s$, $\oG^t_s$ and $\G^{t,x}_s$, we drop the superscript $t$ whenever $t=0$.

Given $x\in\R^d$, we define for any $\tilde{\omega}\in\Omega^t$ the shifted path $(\tilde{\omega}+x)_\cdot:=\tilde{\omega}_\cdot+x$, and for any $A\subseteq\Omega^t$ the shifted set $A+x:=\{\tilde{\omega}\in\Omega^t \mid \tilde{\omega}-x\in A\}$. Then, we define the shifted Wiener measure $\P^{t,x}$ by $\P^{t,x}(F):=\P^t(F-x)$, $F\in\G^t_T$, and let $\overline{\P}^{t,x}$ denote the extension of $\P^{t,x}$ on $(\Omega^t,\oG^t_T)$. For $\P^{t,x}$ and $\overline{\P}^{t,x}$, we drop the superscripts $t$ and $x$ whenever $t=0$ and $x=0$. We let $\E$ denote the expectation taken under $\overline{\P}$.

Fix $t\in [0,T]$ and $\omega\in\Omega$. For any $\tilde{\omega}\in\Omega^t$, we define the concatenation of $\omega$ and $\tilde{\omega}$ at $t$ as 
\[
(\omega\otimes_t\tilde{\omega})_r:=\omega_r1_{[0,t]}(r) + (\tilde{\omega}_r-\tilde{\omega}_t+\omega_t)1_{(t,T]}(r),\ r\in[0,T].
\]
Note that $\omega\otimes_t\tilde{\omega}$ lies in $\Omega$.
Consider the shift operator in space $\psi_t:\Omega^t\mapsto\Omega^t$ defined by $\psi_t(\tilde{\omega}):=\tilde{\omega}-\tilde{\omega}_t$, and the shift operator in time $\phi_t:\Omega\mapsto\Omega^t$ defined by $\phi_t(\omega):=\omega|_{[t,T]}$, the restriction of $\omega\in\Omega$ on $[t,T]$. For any $r\in[t,T]$, since $\psi_t$ and $\phi_t$ are by definition continuous under the norms $\|\cdot\|_{t,r}$ and $\|\cdot\|_{0,r}$ respectively, $\psi_t:(\Omega^t,\G^t_r)\mapsto (\Omega^t,\G^t_r)$ and $\phi_t:(\Omega,\G_r)\mapsto(\Omega^t,\G^t_r)$ are Borel measurable. Then, for any $\xi:\Omega\mapsto\R$, we define the shifted functions $\xi^{t,\omega}:\Omega\mapsto\R$ by
\[
\xi^{t,\omega}(\omega'):=\xi(\omega\otimes_t\phi_t(\omega'))\ \hbox{for}\ \omega'\in\Omega. 
\] 

Given a random time $\tau:\Omega\mapsto[0,\infty]$, whenever $\omega\in\Omega$ is fixed, we simplify our notation as
\[
\omega\otimes_\tau\tilde{\omega}=\omega\otimes_{\tau(\omega)}\tilde{\omega},\ \ \ \ \xi^{\tau,\omega}=\xi^{\tau(\omega),\omega},\ \ \ \  \phi_{\tau}=\phi_{\tau(\omega)},\ \ \ \ \psi_{\tau}=\psi_{\tau(\omega)}.
\]

\begin{defn}\label{defn:F}
On the space $\Omega$, we define, for each $t\in[0,T]$, the filtration $\mathbb{F}^t=\{\F^t_s\}_{s\in[0,T]}$ by
\[
\F^t_s:=\mathcal{J}^t_{s+},\ \hbox{where}\
\mathcal{J}^t_s :=
\begin{cases}
\{\emptyset,\Omega\}, & \hbox{if}\ s\in [0,t],\\
\sigma\left(\phi_t^{-1}\psi_t^{-1}\G^{t,0}_s\cup\overline{\cN}\right), & \hbox{if}\ s\in[t,T].
\end{cases}
\] 
We drop the superscript $t$ whenever $t=0$.
\end{defn}

\begin{rem}
Given $t\in[0,T]$, note that $\F^t_s$ is a collection of subsets of $\Omega$ for each $s\in[0,T]$, whereas $\G^t_s$, $\overline{\G}^t_s$ and $\G^{t,x}_s$ are collections of subsets of $\Omega^t$ for each $s\in[t,T]$.
\end{rem}

\begin{rem}
By definition, $\mathcal{J}_s=\oG_s$ $\forall s\in[0,T]$; then the right continuity of $\overline{\mathbb{G}}$ implies $\F_s=\oG_s$ $\forall s\in[0,T]$ i.e. $\mathbb{F}=\overline{\mathbb{G}}$. Moreover, from Lemma~\ref{lem:measurability} (iii) in Appendix~\ref{sec:appendix} and the right continuity of $\overline{\mathbb{G}}$, we see that $\F^t_s\subseteq\oG_s=\F_s\ \forall s\in[0,T]$, i.e. $\mathbb{F}^t\subseteq\mathbb{F}$. 
\end{rem}

\begin{rem}
Intuitively, $\mathbb{F}^t$ represents the information structure one would have if one starts observing at time $t\in[0,T]$. More precisely, for any $s\in[t,T]$, $\G^{t,0}_s$ represents the information structure one obtains after making observations on $W^t$ in the period $[t,s]$. One could then deduce from $\G^{t,0}_s$ the information structure $\phi_t^{-1}\psi_t^{-1}\G^{t,0}_s$ for $W$ on the interval $[0,s]$.
\end{rem}

We define $\T^t$ as the set of all $\mathbb{F}^t$-stopping times which take values in $[0,T]$ $\overline{\P}$-a.s., and $\A_t$ as the set of all $\mathbb{F}^t$-progressively measurable $M$-valued processes, where $M$ is a separable metric space. Also, for any $\mathbb{F}$-stopping times $\tau_1,\tau_2$ with $\tau_1\le\tau_2$ $\overline{\P}$-a.s., we denote by $\T^t_{\tau_1,\tau_2}$ the set of all $\tau\in\T^t$ which take values in $[\tau_1,\tau_2]$ $\overline{\P}$-a.s. Again, we drop the sub- or superscript $t$ whenever $t=0$.

\subsection{The State Process}
Given $(t,x)\in[0,T]\times\R^d$ and $\alpha\in\A$, let $X^{t,x,\alpha}$ denote a $\mathbb{R}^d$-valued process satisfying the following SDE:
\begin{equation}\label{SDE X}
dX^{t,x,\alpha}_s=b(s,X^{t,x,\alpha}_s,\alpha_s)ds+\sigma(s,X^{t,x,\alpha}_s,\alpha_s)dW_s,\ \ \ s\in[t,T],
\end{equation}
with the initial condition $X_{t}^{t,x,\alpha}=x$. 
Let $\mathbb{M}^d$ be the set of $d\times d$ real matrices. We assume that $b:[0,T]\times\R^d\times M\mapsto\R^d$ and $\sigma:[0,T]\times\R^d\times M\mapsto \mathbb{M}^d$ are deterministic Borel functions, and $b(t,x,u)$ and $\sigma(t,x,u)$ are continuous in $(x,u)$; moreover,
there exists $K>0$ such that for any $t\in[0,T],\ x,y\in\mathbb{R}^d$, and $u\in M$, 
\begin{eqnarray}
&&|b(t,x,u)-b(t,y,u)|+|\sigma(t,x,u)-\sigma(t,y,u)| \leq K|x-y|, \label{glo Lips 1}\\
&&|b(t,x,u)|+|\sigma(t,x,u)| \leq K (1+|x|).\label{linear grow}
\end{eqnarray}
The conditions above imply that: for any initial condition $(t,x)\in[0,T]\times\mathbb{R}^d$ and control $\alpha\in\mathcal{A}$, \eqref{SDE X} admits a unique strong solution $X^{t,x,\alpha}_\cdot$. Moreover, without loss of generality, we define
\begin{equation}\label{X=x}
X^{t,x,\alpha}_s:=x\ \ \ \ \hbox{for}\  s<t. 
\end{equation}

\begin{rem}\label{rem:est}
Fix $\alpha\in\A$. Under \eqref{glo Lips 1} and \eqref{linear grow}, the same calculations in \cite[Appendix]{Pham98} and \cite[Proposition 1.2.1]{Bouchard-note-07} yield the following estimates: for each $p\ge 1$, there exists $C_p(\alpha)>0$ such that for any $(t,x), (t',x')\in[0,T]\times\R^d$, and $h\in[0,T-t]$,
\begin{align}
\E\left[\sup_{\ms 0\le s\le T}|X^{t,x,\alpha}_s|^p\right]&\le C_p(1+|x|^p);\label{est 1}\\
\E\left[\sup_{\ms 0\le s\le t+h}|X^{t,x,\alpha}_s-x|^p\right]&\le C_ph^{\frac{p}{2}}(1+|x|^p);\label{est 2}\\
\E\left[\sup_{\ms 0\le s\le T}|X^{t',x',\alpha}_s-X^{t,x,\alpha}_s|^p\right]&\le C_p\left[|x'-x|^p+|t'-t|^{\frac{p}{2}}(1+|x|^p)\right]\label{est 3}.
\end{align}
\end{rem} 

\begin{rem}[{\bf flow property}]\label{rem:flow property}
By pathwise uniqueness of the solution to \eqref{SDE X}, for any $0\le t\le s\le T$, $x\in\R^d$, and $\alpha\in\A$, we have the following two properties:
\begin{itemize}
\item [(i)] $X^{t,x,\alpha}_r(\omega)= X^{s,X^{t,x,\alpha}_s,\alpha}_r(\omega)$ $\forall\ r\in[s,T]$, for $\overline{\P}$-a.e. $\omega\in\Omega$; see \cite[Chapter 2]{Bouchard-note-07} and \cite[p.41]{Pham-book}.
\item [(ii)] By (1.16) in \cite{FS89} and the discussion below it, for $\overline{\P}$-a.e. $\omega\in\Omega$, we have
\[
X^{t,x,\alpha}_r \left(\omega\otimes_s\phi_s(\omega')\right)= X^{s,X^{t,x,\alpha}_s(\omega),\alpha^{s,\omega}}_r\left(\omega'\right)\ \forall r\in[s,T],\ \hbox{for}\ \overline{\P}\hbox{-a.e}\ \omega'\in\Omega;
\]
see also \cite[Lemma 3.3]{Nutz12-Quasi-sure}.
\end{itemize}
\end{rem}

\subsection{Properties of Shifted Objects} Let us first derive some properties of $\F^t_T$-measurable random variables.

\begin{prop}\label{prop:xi indep. of F_t}
Fix $t\in[0,T]$ and $\xi\in L^0(\Omega,\F^t_T)$.
\begin{itemize}
\item [(i)] $\F^t_T$ and $\F_t$ are independent. This in particular implies that $\xi$ is independent of $\F_t$.
\item [(ii)] There exist $\overline{N}, \overline{M}\in\overline{\cN}$ such that: for any fixed $\omega\in\Omega\setminus\overline{N}$, $\xi^{t,\omega}(\omega')=\xi(\omega')\ \forall\omega'\in\Omega\setminus\overline{M}$. 
\end{itemize}
\end{prop}

\begin{proof}
See Appendix~\ref{subsec:xi indep. of F_t}.
\end{proof}

Fix $\theta\in\T$. Given $\alpha\in\A$, we can define, for $\overline{\P}$-a.e. $\omega\in\Omega$, a control $\alpha^{\theta,\omega}\in\A_{\theta(\omega)}$ by  
\[
\alpha^{\theta,\omega}(\omega'):=\{\alpha_r^{\theta,\omega}(\omega')\}_{r\in[0,T]}=\left\{\alpha_r\left(\omega\otimes_{\theta}\phi_{\theta}(\omega')\right)\right\}_{r\in[0,T]},\ \omega'\in\Omega;
\]
see \cite[proof of Proposition 5.4]{BT11}. Here, we state a similar result for stopping times in $\T$.

\begin{prop}\label{lem: tau a stopping time a.s.}
Fix $\theta\in\T$. For any $\tau\in\T_{\theta,T}$, we have $\tau^{\theta,\omega}\in\T^{\theta(\omega)}_{\theta(\omega),T}$ for $\overline{\P}$-a.e. $\omega\in\Omega$.
\end{prop}

\begin{proof}
See Appendix~\ref{subsec: tau a stopping time a.s.}.
\end{proof}

Let $\rho:M\times M\mapsto\R$ be any given metric on $M$. By \cite[p.142]{Krylov-book-80}, $\rho'(u,v):=\frac{2}{\pi}\arctan\rho(u,v)<1$ for $u,v\in M$ is a metric equivalent to $\rho$, from which we can construct a metric on $\A$ by
\begin{equation}\label{metric on A}
\tilde{\rho}(\alpha,\beta):=\E\left[\int_0^T\rho'(\alpha_t,\beta_t)dt\right]\ \hbox{for}\ \alpha,\beta\in\A.
\end{equation}   

Now, we state a generalized version of Proposition~\ref{prop:xi indep. of F_t} (ii) for controls $\alpha\in\A$. 

\begin{prop}\label{prop:alpha^t,w=alpha}
Fix $t\in[0,T]$ and $\alpha\in\A_t$. There exists $\overline{N}\in\overline{\cN}$ such that: for any $\omega\in\Omega\setminus\overline{N}$, $\tilde{\rho}(\alpha^{t,\omega},\alpha)=0$. Furthermore, for any $(s,x)\in[0,T]\times\R^d$, $X^{s,x,\alpha^{t,\omega}}_r(\omega')=X^{s,x,\alpha}_r(\omega')$, $r\in[s,T]$, for $\overline{\P}$-a.e. $\omega'\in\Omega$.     
\end{prop}

\begin{proof}
See Appendix~\ref{subsec:alpha^t,w=alpha}.
\end{proof}


\section{Problem Formulation}\label{sec:problem}
We consider a controller-and-stopper game under the finite time horizon $T>0$. While the controller has the ability to affect the state process $X^\alpha$ through the selection of the control $\alpha$, the stopper has the right to choose the duration of this game, in the form of a random time $\tau$. Suppose the game starts at time $t\in[0,T]$. The stopper incurs the running cost $f(s,X^{\alpha}_s,\alpha_s)$ at every moment $t\le s<\tau$, and the terminal cost $g(X^\alpha_\tau)$ at the time the game stops, where $f$ and $g$ are some given deterministic functions.
According to the instantaneous discount rate $c(s,X^{\alpha}_s)$ for some given deterministic function $c$, the two players interact as follows: the stopper would like to stop optimally so that her expected discounted cost could be minimized, whereas the controller intends to act adversely against her by manipulating the state process $X^\alpha$ in a way that frustrates the effort of the stopper. 

For any $t\in[0,T]$, there are two possible scenarios for this game. In the first scenario, the stopper acts first. At time $t$, while the stopper is allowed to use the information of the path of $W$ up to time $t$ for her decision making, the controller has advantage: she has access to not only the path of $W$ up to $t$ but also the stopper's decision. Choosing one single stopping time, as a result, might not be optimal for the stopper. Instead, she would like to employ a stopping {\it strategy} which will give different responses to different future actions the controller will take. 


\begin{defn}\label{defn:strategy}
Given $t\in[0,T]$, we say a function $\pi:\A\mapsto\T_{t,T}$ is an admissible stopping strategy on the horizon $[t,T]$ if it satisfies the following conditions:
\begin{itemize}
\item [(i)] for any $\alpha,\beta\in\A$, it holds for $\overline{\P}$-a.e. $\omega\in\Omega$ that
\begin{equation}\label{strategy 1}
\begin{split}
&\hbox{if}\ \min\{\pi[\alpha](\omega),\pi[\beta](\omega)\}\le \inf\left\{s\ge t\ \middle|\ \int_t^s\rho'(\alpha_r(\omega),\beta_r(\omega))dr \neq 0\right\},\\
&\hspace{0.2in}\hbox{then}\ \pi[\alpha](\omega)=\pi[\beta](\omega).
\end{split}
\end{equation}
Recall that $\rho'$ is a metric on $M$ defined right above \eqref{metric on A}.
\item [(ii)] for any $s\in[0,t]$, if $\alpha\in\A_s$, then $\pi[\alpha]\in\T^s_{t,T}$. 
\item [(iii)] for any $\alpha\in\A$ and $\theta\in\T$ with $\{\theta\le t\}\notin\overline{\cN}$, it holds for $\overline{\P}$-a.e. $\omega\in\{\theta\le t\}$ that
\[
\pi[\alpha]^{\theta,\omega}(\omega')=\pi[\alpha^{\theta,\omega}](\omega'),\ \hbox{for}\ \overline{\P}\hbox{-a.e.}\ \omega'\in\Omega.
\]
\end{itemize}
We denote by $\Pi_{t,T}$ the set of all admissible stopping strategies on the horizon $[t,T]$.
\end{defn}

\begin{rem}
Definition~\ref{defn:strategy} (i) serves as the non-anticipativity condition for the stopping strategies. The intuition behind it should be clear: Suppose we begin our observation at time $t$, and employ a strategy $\pi\in\Pi_{t,T}$. By taking the control $\alpha$ and following the path $\omega$, we decide to stop at the moment $\pi[\alpha](\omega)$. If, up to this moment, we actually cannot distinguish between the controls $\alpha$ and $\beta$ , then we should stop at the same moment if we were taking the control $\beta$. 

Moreover, as shown in Proposition~\ref{prop:equivalence} below, \eqref{strategy 1} is equivalent to the following statement:
\begin{equation}\label{strategy 2}
\hbox{For any}\ \alpha,\beta\in\A\ \hbox{and}\ s\in[t,T],\ 1_{\{\pi[\alpha]\le s\}}=1_{\{\pi[\beta]\le s\}}\ \hbox{for}\ \overline{\P}\hbox{-a.e.}\ \omega\in\{\alpha =_{[t,s)} \beta\},
\end{equation}
where $\{\alpha =_{[t,s)} \beta\}:=\{\omega\in\Omega\mid \alpha_r(\omega)=\beta_r(\omega)\ \hbox{for}\ \hbox{a.e.}\ r\in[t,s)\}$. This shows that Definition~\ref{defn:strategy} (i) extends the non-anticipativity of strategies from two-controller games (see e.g. \cite{BL08}) to current context of controller-and-stopper games. 

Also notice that \eqref{strategy 2} is similar to, yet a bit weaker than, Assumption (C5) in \cite{BMN12}. This is because in the definition of $\{\alpha =_{[t,s)} \beta\}$, \cite{BMN12} requires $\alpha_r=\beta_r$ for all, instead of almost every, $r\in[t,s)$.
\end{rem}

\begin{prop}\label{prop:equivalence}
Fix $t\in[0,T]$. For any function $\pi:\A\mapsto\T_{t,T}$, \eqref{strategy 1} holds iff \eqref{strategy 2} holds.
\end{prop}

\begin{proof}
For any $\alpha,\beta\in\A$, we set $\theta(\omega):=\inf\{s\ge t\mid \int_t^s\rho'(\alpha_r(\omega),\beta_r(\omega))dr \neq 0\}$.

{\bf Step 1:} Suppose $\pi$ satisfies \eqref{strategy 1}. For any $\alpha,\beta\in\A$, take some $N\in\overline{\cN}$ such that \eqref{strategy 1} holds for $\omega\in\Omega\setminus N$. Fix $s\in[t,T]$. Given $\omega\in\{\alpha =_{[t,s)} \beta\}\setminus N$, we have $s \le\theta(\omega)$. If $\pi[\alpha](\omega)\le\theta(\omega)$, then \eqref{strategy 1} implies $\pi[\alpha](\omega)=\pi[\beta](\omega)$, and thus $1_{\{\pi[\alpha]\le s\}}(\omega)=1_{\{\pi[\beta]\le s\}}(\omega)$. If $\pi[\alpha](\omega)>\theta(\omega)$, then \eqref{strategy 1} implies $\pi[\beta](\omega)>\theta(\omega)$ too. It follows that $1_{\{\pi[\alpha]\le s\}}(\omega)=0=1_{\{\pi[\beta]\le s\}}(\omega)$, since $s \le\theta(\omega)$. This already proves \eqref{strategy 2}.

{\bf Step 2:} Suppose \eqref{strategy 2} holds. Fix $\alpha,\beta\in\A$. By \eqref{strategy 2}, there exists some $N\in\overline{\cN}$ such that 
\begin{equation}\label{(i)'}
\hbox{for any}\ s\in\Q\cap[t,T],\ 1_{\{\pi[\alpha]\le s\}}=1_{\{\pi[\beta]\le s\}}\ \hbox{for}\ \omega\in\{\alpha =_{[t,s)} \beta\}\setminus N.
\end{equation}
Fix $\omega\in\Omega\setminus N$. For any $s\in\Q\cap[t,\theta(\omega)]$, we have $\omega\in\{\alpha =_{[t,s)} \beta\}$. Then \eqref{(i)'} yields 
\begin{equation}\label{(i)''}
1_{\{\pi[\alpha]\le s\}}(\omega)=1_{\{\pi[\beta]\le s\}}(\omega),\ \hbox{for all}\ s\in\Q\cap[t,\theta(\omega)].
\end{equation}
If $\pi[\alpha](\omega)\le\theta(\omega)$, take an increasing sequence $\{s_n\}_{n\in\N}\subset\Q\cap[t,\theta(\omega)]$ such that $s_n\uparrow \pi[\alpha](\omega)$. Then \eqref{(i)''} implies $\pi[\beta](\omega)> s_n$ for all $n$, and thus $\pi[\beta](\omega)\ge \pi[\alpha](\omega)$. Similarly, by taking a decreasing sequence $\{r_n\}_{n\in\N}\subset\Q\cap[t,\theta(\omega)]$ such that $r_n\downarrow \pi[\alpha](\omega)$, we see from \eqref{(i)''} that $\pi[\beta]\le r_n$ for all $n$, and thus $\pi[\beta](\omega)\le \pi[\alpha](\omega)$. We therefore conclude $\pi[\beta](\omega)=\pi[\alpha](\omega)$. Now, if $\pi[\beta](\omega)\le\theta(\omega)$, we may argue as above to show that $\pi[\alpha](\omega)=\pi[\beta](\omega)$. This proves \eqref{strategy 1}. 
\end{proof}

Next, we give concrete examples of strategies under Definition~\ref{defn:strategy}. 

\begin{eg}
Given $t\in[0,T]$, define $\lambda_t:\Omega\mapsto\Omega$ by $(\lambda_t(\omega))_\cdot:=\omega_{\cdot\wedge t}$. Recall the space $C([t,T];\R^d)$ of continuous functions mapping $[t,T]$ into $\R^d$. For any $x\in\R^d$, we define $\pi:\A\mapsto\T_{t,T}$ by 
\begin{equation}\label{S(X)}
\pi[\alpha](\omega):= S\left(\{X^{t,x,\alpha}_r(\omega)\}_{r\in[t,T]}\right),
\end{equation}
for some function $S:C([t,T];\R^d)\mapsto[t,T]$ satisfying $\{\xi\mid S(\xi)\le s\}\in\lambda_s^{-1}\mathcal{X}^t_T$ $\forall s\in[t,T]$, where $\mathcal{X}^t_T$ denotes the Borel $\sigma$-algebra generated by $C([t,T];\R^d)$. Note that the formulation \eqref{S(X)} is similar to the stopping rules introduced in the one-dimensional controller-and-stopper game in \cite{KS01}, and it covers concrete examples such as exit strategies of a Borel set (see e.g. \eqref{def pi^circ} below). We claim that Definition~\ref{defn:strategy} readily includes the formulation \eqref{S(X)}.

Let the function $\pi:\A\mapsto\T_{t,T}$ be given as in \eqref{S(X)}. First, for any $\alpha,\beta\in\A$, set $\theta:=\inf\{s\ge t\mid \int_t^s\rho'(\alpha_r(\omega),\beta_r(\omega))dr \neq 0\}$. Observing that the strong solutions $X^{t,x,\alpha}$ and $X^{t,x,\beta}$ coincide on the interval $[t,\theta)$ $\overline{\P}$-a.s., we conclude that $\pi$ satisfies Definition~\ref{defn:strategy} (i). Next, for any $s\in[0,t]$, since $X^{t,x,\alpha}$ depends on $\F_s$ only through the control $\alpha$, Definition~\ref{defn:strategy} (ii) also holds for $\pi$. To check Definition~\ref{defn:strategy} (iii), let us introduce, for any $\theta\in\T$ with $\{\theta\le t\}\notin\overline{\cN}$, the strong solution $\tilde{X}$ to the SDE \eqref{SDE X} with the drift coefficient $\tilde{b}(s,x,u):=1_{\{s<t\}}0+1_{\{s\ge t\}}b(s,x,u)$ and the diffusion coefficient $\tilde{\sigma}(s,x,u):=1_{\{s<t\}}0+1_{\{s\ge t\}}\sigma(s,x,u)$. Then, by using the pathwise uniqueness of strong solutions and Remark~\ref{rem:flow property} (ii), for $\overline{\P}$-a.e. $\omega\in\{\theta\le t\}$,
\[
X^{t,x,\alpha}_r(\omega\otimes_\theta\phi_\theta(\omega'))=\tilde{X}^{0,x,\alpha}_r(\omega\otimes_\theta\phi_\theta(\omega'))=\tilde{X}^{\theta(\omega),\tilde{X}^{0,x,\alpha}_\theta(\omega),\alpha^{\theta,\omega}}_r(\omega')=\tilde{X}_r^{\theta(\omega),x,\alpha^{\theta,\omega}}(\omega')=X_r^{t,x,\alpha^{\theta,\omega}}(\omega'),
\] 
$\forall r\in[t,T]$, for $\overline{\P}$-a.e. $\omega'\in\Omega$. This implies 
\[
\pi[\alpha]^{\theta,\omega}(\omega')=S(\{X^{t,x,\alpha}_r(\omega\otimes_\theta\phi_\theta(\omega'))\}_{r\in[t,T]})=S(\{X^{t,x,\alpha^{\theta,\omega}}_r(\omega')\}_{r\in[t,T]})=\pi[\alpha^{\theta,\omega}](\omega'),
\]
for $\overline{\P}$-a.e. $\omega'\in\Omega$, which is Definition~\ref{defn:strategy} (iii).
\end{eg}



Let us now look at the second scenario in which the controller acts first. In this case, the stopper has access to not only the path of $W$ up to time $t$ but also the controller's decision. The controller, however, does not use strategies as an attempt to offset the advantage held by the stopper. As the next remark explains, the controller merely chooses one single control because she would not benefit from using non-anticipating strategies.

\begin{rem}\label{rem:no control strategy}
Fix $t\in[0,T]$. Let $\gamma:\T\mapsto\A_t$ satisfy the following non-anticipativity condition: for any $\tau_1,\tau_2\in\T$ and $s\in[t,T]$, it holds for $\overline{\P}$-a.e. $\omega\in\Omega$ that 
\[
\hbox{if}\ \min\{\tau_1(\omega),\tau_2(\omega)\}>s,\ \hbox{then}\ (\gamma[\tau_1])_r(\omega)=(\gamma[\tau_2])_r(\omega)\ \hbox{for}\ r\in[t,s).
\]
Then, observe that $\gamma[\tau](\omega)=\gamma[T](\omega)$ on $[t,\tau(\omega))$ $\overline{\P}$-a.s. for any $\tau\in\T$. This implies that employing the strategy $\gamma$ has the same effect as employing the control $\gamma[T]$. In other words, the controller would not benefit from using non-anticipating strategies.  
\end{rem}

Now, we are ready to introduce the upper and lower value functions of the game of control and stopping. For $(t,x)\in[0,T]\times\R^d$, if the stopper acts first, the associated value function is given by 
\begin{equation}\label{U}
U(t,x):=\inf_{\pi\in\Pi_{t,T}}\sup_{\alpha\in\A_t}\E\bigg[\int_{t}^{\pi[\alpha]}e^{-\int_{t}^{s}c(u,X_{u}^{t,x,\alpha})du}f(s,X_{s}^{t,x,\alpha},\alpha_s)ds + e^{-\int_{t}^{\pi[\alpha]}c(u,X_{u}^{t,x,\alpha})du}g(X_{\pi[\alpha]}^{t,x,\alpha})\bigg].
\end{equation} 
On the other hand, if the controller acts first, the associated value function is given by 
\begin{equation}\label{V}
V(t,x):=\sup_{\alpha\in\mathcal{A}_t}\inf_{\tau\in\mathcal{T}^t_{t,T}}\E\bigg[\int_{t}^{\tau}e^{-\int_{t}^{s}c(u,X_{u}^{t,x,\alpha})du}f(s,X_{s}^{t,x,\alpha},\alpha_s)ds + e^{-\int_{t}^{\tau}c(u,X_{u}^{t,x,\alpha})du}g(X_{\tau}^{t,x,\alpha})\bigg].
\end{equation} 
By definition, we have $U\ge V$. We therefore call $U$ the upper value function, and $V$ the lower value function. We say the game has a value if these two functions coincide.

\begin{rem}
In a game with two controllers (see e.g. \cite{EK74-TAMS, EK74-JMAA, FS89, BL08}), upper and lower value functions are also introduced. However, since both of the controllers use strategies, it is difficult to tell, just from the definitions, whether one of the value functions is larger than the other (despite their names). In contrast, in a controller-stopper game, only the stopper uses strategies, thanks to Remark~\ref{rem:no control strategy}. We therefore get $U\ge V$ for free, which turns out to be a crucial relation in the PDE characterization for the value of the game.    
\end{rem}

We assume that the cost functions $f, g$ and the discount rate $c$ satisfy the following conditions: $f:[0,T]\times\R^d\times M\mapsto\R_+$ is Borel measurable, and $f(t,x,u)$ is continuous in $(x,u)$, and continuous in $x$ uniformly in $u\in M$ for each $t$; $g:\R^d\mapsto\R_+$ is continuous; $c:[0,T]\times\R^d\mapsto\R_+$ is continuous and bounded above by some real number $\bar{c}>0$. Moreover, we impose the following polynomial growth condition on $f$ and $g$ 
\begin{equation}\label{poly grow}
|f(t,x,u)|+|g(x)|\leq K(1+|x|^{\bar{p}})\ \text{for some}\ \bar{p} \ge 1.
\end{equation}

\begin{rem}\label{rem:eps optimal}
Presumably, by imposing additional assumptions, one could construct a saddle point of optimal choices for a controller-and-stopper game. For example, in the one-dimensional game in \cite{KS01}, a saddle point is constructed under additional assumptions on the cost function and the dynamics of the state process (see (6.1)-(6.3) in \cite{KS01}). For the multi-dimensional case, in order to find a saddle point, \cite{KZ08} assumes that the cost function and the drift coefficient are continuous with respect to the control variable, and the associated Hamiltonian always attains its infimum (see (71)-(73) in \cite{KZ08}); whereas \cite{Hamadene06} and \cite{HL00} require compactness of the control set.

In this paper, we have no plan to impose additional assumptions for constructing saddle points. Instead, we intend to investigate, under a rather general set-up, whether the game has a value and how we can characterize this value if it exists.    
\end{rem}

\begin{rem}
For any $(t,x)\in[0,T]\times\R^d$ and $\alpha\in\A$, the polynomial growth condition \eqref{poly grow} and \eqref{est 1} imply that
\begin{equation}\label{dominated}
\E\bigg[\sup_{t\leq r\leq T}\left(\int_{t}^{r}e^{-\int_{t}^{s}c(u,X_{u}^{t,x,\alpha})du}f(s,X_{s}^{t,x,\alpha},\alpha_s)ds + e^{-\int_{t}^{r}c(u,X_{u}^{t,x,\alpha})du}g(X_{r}^{t,x,\alpha})\right) \bigg]<\infty.
\end{equation}
\end{rem}

\begin{lem}\label{lem:convergence}
Fix $\alpha\in\A$ and $(s,x)\in[0,T]\times\R^d$. For any $\{(s_n,x_n)\}_{n\in\N}\subset[0,T]\times\R^d$ such that $(s_n,x_n)\to (s,x)$, we have
\begin{align}
&\E\bigg[\sup_{0\le r\le T}|g(X^{s_n,x_n,\alpha}_r)-g(X^{s,x,\alpha}_r)|\bigg]\to 0;\label{Y conti. 1}\\
&\E\int_0^T|1_{[s_n,T]}(r)f(r,X^{s_n,x_n,\alpha}_r,\alpha_r)-1_{[s,T]}(r)f(r,X^{s,x,\alpha}_r,\alpha_r)|dr\to 0\label{Y conti. 2}.
\end{align}
\end{lem}

\begin{proof}
In view of \eqref{est 3}, we have, for any $p\ge 1$, 
\begin{equation}\label{est 3'}
\E\bigg[\sup_{\ms 0\le r\le T}|X^{s_n,x_n,\alpha}_r-X^{s,x,\alpha}_r|^p\bigg]\to 0.
\end{equation}
Thanks to the above convergence and the polynomial growth condition \eqref{poly grow} on $f$, we observe that \eqref{Y conti. 2} is a consequence of \cite[Lemma 2.7.6]{Krylov-book-80}. 

It remains to prove \eqref{Y conti. 1}. Fix $\eps$, $\eta>0$. Take $a>0$ large enough such that $\frac{2 C_1T(2+|x|)}{a}<\frac{\eta}{3}$, where $C_1>0$ is given as in Remark~\ref{rem:est}. Since $g$ is continuous, it is uniformly continuous on $\bar{B}_a(x):=\{y\in\R^d\mid |y-x|\le a\}$. Thus, there exists some $\delta>0$ such that $|g(x)-g(y)|<\eps$ for all $x,y\in\bar{B}_a(x)$ with $|x-y|<\delta$. Define 
\begin{align*}
&A:=\bigg\{\sup_{\ms 0\le r\le T}|X^{s,x,\alpha}_r-x|>a\bigg\},\ \ \ B_n:=\bigg\{\sup_{\ms 0\le r\le T}|X^{s_n,x_n,\alpha}_r-x|>a\bigg\},\\
&B'_n:=\bigg\{\sup_{\ms 0\le r\le T}|X^{s_n,x_n,\alpha}_r-x_n|>\frac{a}{2}\bigg\},\ \ \ D_n:=\bigg\{\sup_{\ms 0\le r\le T}|X^{s_n,x_n,\alpha}_r-X^{s,x,\alpha}_r|\ge \delta\bigg\}.
\end{align*}
By the Markov inequality and \eqref{est 2},   
\begin{equation*}
\overline{\P}(A)\le \frac{C_1\sqrt{T}(1+|x|)}{a}<\frac{\eta}{3},\ \ \ \ \overline{\P}(B'_n)\le \frac{2 C_1\sqrt{T}(1+|x_n|)}{a}<\frac{\eta}{3}\ \ \hbox{for}\ n\ \hbox{large enough}.
\end{equation*}
On the other hand, \eqref{est 3'} implies that $\overline{\P}(D_n)< \frac{\eta}{3}$ for $n$ large enough. Noting that $(B'_n)^c\subseteq B^c_n$ for $n$ large enough, we obtain
\begin{align*}
\overline{\P}\bigg(\sup_{\ms 0\le r\le T}|g(X^{s_n,x_n,\alpha}_r)-g(X^{s,x,\alpha}_r)| > \eps\bigg) &\le 1-\overline{\P}(A^c\cap B_n^c\cap D_n^c)=\overline{\P}(A\cup B_n\cup D_n)\\
&\le \overline{\P}(A\cup B'_n\cup D_n)<\eta,\ \ \ \hbox{for}\ n\ \hbox{large enough}.
\end{align*}
Thus, we have $h_n:=\sup_{\ms 0\le r\le T}|g(X^{s_n,x_n,\alpha}_r)-g(X^{s,x,\alpha}_r)|\to 0$ in probability. Finally, observing that the polynomial growth condition \eqref{poly grow} on $g$ and \eqref{est 1} imply that $\{h_n\}_{n\in\N}$ is $L^2$-bounded, we conclude that $h_n\to 0$ in $L^1$, which gives \eqref{Y conti. 1}.   
\end{proof}

\subsection{The Associated Hamiltonian} For $(t,x,p,A)\in[0,T]\times\mathbb{R}^d\times\mathbb{R}^d\times\mathbb{M}^d$,  we associate the following Hamiltonian with our mixed control/stopping problem:
\begin{equation}\label{H}
H(t,x,p,A):=\inf_{a\in M}H^a(t,x,p,A),
\end{equation}
where 
\[
H^a(t,x,p,A):=-b(t,x,a)\cdot p-\frac{1}{2}Tr[\sigma\sigma'(t,x,a)A]-
f(t,x,a).
\]
Since $b$, $\sigma$, and $f$ are assumed to be continuous only in $(x,a)$, and $M$ is a separable metric space without any compactness assumption, the operator $H$ may be neither upper nor lower semicontinuous.
As a result, we will need to consider an upper semicontinuous version of $H$ defined by 
\begin{equation}\label{bH}
\overline{H}(t,x,p,A):=\inf_{a\in M} (H^a)^*(t,x,p,A),
\end{equation}
where $(H^a)^*$ is the upper semicontinuous envelope of $H^a$, defined as in \eqref{envelopes}; see Proposition ~\ref{prop:vis super}. On the other hand, we will need to consider the lower semicontinuous envelope $H_*$, defined as in \eqref{envelopes}, in Proposition ~\ref{prop:vis sub}. Notice that $\overline{H}$ is different from the upper semicontinuous envelope $H^*$, defined as in \eqref{envelopes} (in fact, $\overline{H}\ge H^*$). See Remark~\ref{rem:why bH} for our choice of $\overline{H}$ over $H^*$.

\subsection{Reduction to the Mayer Form} Given $t\in[0,T]$ and $\alpha\in\A_t$, let us increase the state process to $(X,Y,Z)$, where
\begin{align*}
dY^{t,x,y,\alpha}_s &= -Y^{t,x,y,\alpha}_s c(s,X^{t,x,\alpha}_s)ds,\ s\in[t,T],\ \hbox{with}\ Y^{t,x,y,\alpha}_t=y\ge 0;\\
Z^{t,x,y,z,\alpha}_s &:= z + \int_{t}^s Y^{t,x,y,\alpha}_r f(r,X^{t,x,\alpha}_r,\alpha_r)dr,\ \ \hbox{for some}\ z\ge 0.
\end{align*} 
Set $\mathcal{S}:= \R^d\times\R_+\times\R_+$. For any $\bx:=(x,y,z)\in\mathcal{S}$, we define
\[\bX^{t,\bx,\alpha}_s:=
\left(\begin{array}{c}
X^{t,x,\alpha}_s\\Y^{t,x,y,\alpha}_s\\Z^{t,x,y,z,\alpha}_s\\
\end{array}\right),\]
and consider the function $F:\mathcal{S}\mapsto\R_+$ defined by
\[
F(x,y,z):=z+yg(x).
\]
Now, we introduce the functions $\bar{U},\bar{V}:[0,T]\times\mathcal{S}\mapsto\R$ defined by 
\begin{align*}
\bar{U}(t,x,y,z)&:=\inf_{\pi\in\Pi_{t,T}}\sup_{\alpha\in\mathcal{A}_t}\E\left[F(X^{t,x,\alpha}_{\pi[\alpha]}, Y^{t,x,y,\alpha}_{\pi[\alpha]}, Z^{t,x,y,z,\alpha}_{\pi[\alpha]}) \right]=\inf_{\pi\in\Pi_{t,T}}\sup_{\alpha\in\mathcal{A}_t}\E\left[F(\bX^{t,\bx,\alpha}_{\pi[\alpha]})\right],\\
\bar{V}(t,x,y,z)&:=\sup_{\alpha\in\mathcal{A}_t}\inf_{\tau\in\mathcal{T}^t_{t,T}}\E\left[F(X^{t,x,\alpha}_\tau, Y^{t,x,y,\alpha}_\tau, Z^{t,x,y,z,\alpha}_\tau) \right]= \sup_{\alpha\in\mathcal{A}_t}\inf_{\tau\in\mathcal{T}^t_{t,T}}\E[F(\bX^{t,\bx,\alpha}_\tau)].
\end{align*}
Given $\tau\in\T_{t,T}$, consider the function
\begin{equation}\label{def J}
J(t,\bx;\alpha,\tau):=\E[F(\bX^{t,\bx,\alpha}_\tau)].
\end{equation}
Observing that $F(\bX^{t,\bx,\alpha}_\tau)=z+yF(\bX^{t,x,1,0,\alpha}_\tau)$, we have 
\begin{equation}\label{J = z+yJ}
J(t,\bx;\alpha,\tau)=z+yJ(t,(x,1,0);\alpha,\tau), 
\end{equation}
which in particular implies
\begin{equation}\label{V bV}
\bar{U}(t,x,y,z)=z+yU(t,x)\ \ \ \bar{V}(t,x,y,z)=z+yV(t,x).
\end{equation}
Thus, we can express the value functions $U$ and $V$ as 
\[
U(t,x)=\inf_{\pi\in\Pi_{t,T}}\sup_{\alpha\in\A_t}J(t,(x,1,0);\alpha,\pi[\alpha]),\ \ \ V(t,x)=\sup_{\alpha\in\A_t}\inf_{\tau\in\T^t_{t,T}}J(t,(x,1,0);\alpha,\tau).
\] 

The following result will be useful throughout this paper.

\begin{lem}\label{lem:E = J}
Fix $(t,\bx)\in[0,T]\times\mathcal{S}$ and $\alpha\in\A$. 
For any $\theta\in\T_{t,T}$ and $\tau\in\T_{\theta,T}$, we have
\begin{equation*}
\E[F(\bX^{t,\bx,\alpha}_\tau)  \mid \mathcal{F}_\theta](\omega)=J\left(\theta(\omega),\bX^{t,\bx,\alpha}_{\theta}(\omega);\alpha^{\theta,\omega},\tau^{\theta,\omega}\right),\ \hbox{for}\ \overline{\P}\hbox{-a.e.}\ \omega\in\Omega.
\end{equation*}
\end{lem}

\begin{proof}
See Appendix~\ref{subsec:E = J}.
\end{proof}

\section{Supersolution Property of $V$}\label{sec:super}
In this section, we will first study the following two functions
\begin{equation}\label{G}
G^\alpha(s,\bx):=\inf_{\tau\in\mathcal{T}^s_{s,T}}J(s,\bx;\alpha,\tau),\ \ \ \widetilde{G}^\alpha(s,\bx):=\inf_{\tau\in\T_{s,T}}J(s,\bx;\alpha,\tau),\ \ \ \hbox{for}\ (s,\bx)\in[0,T]\times\mathcal{S},
\end{equation} 
where $\alpha\in\A$ is being fixed. A continuity result of $G^\alpha$ enables us to adapt the arguments in \cite{BT11} to current context. We therefore obtain a weak dynamic programming principle (WDPP) for the function $V$ (Proposition~\ref{prop:WDPP_super}), which in turn leads to the supersolution property of $V$ (Proposition~\ref{prop:vis super}).

\begin{lem}\label{lem:conti} 
Fix $\alpha\in\A$.
\begin{itemize}
\item [(i)]$\widetilde{G}^\alpha$ is continuous on $[0,T]\times\mathcal{S}$.
\item [(ii)] Suppose $\alpha\in\A_t$ for some $t\in[0,T]$. Then $G^\alpha=\widetilde{G}^\alpha$ on $[0,t]\times\mathcal{S}$. As a result, $G^\alpha$ is continuous on $[0,t]\times\mathcal{S}$.
\end{itemize}
\end{lem}

\begin{proof}
{\bf (i)} For any $s\in[0,T]$ and $\bx=(x,y,z)\in\mathcal{S}$, observe from \eqref{J = z+yJ} that $\widetilde{G}^\alpha(s,\bx)=z+y\widetilde{G}^{\alpha}(s,(x,1,0))$. Thus, it is enough to prove that $\widetilde{G}^{\alpha}(s,(x,1,0))$ is continuous on $[0,T]\times\R^d$. Also note that under \eqref{X=x}, we have 
\[
\widetilde{G}^{\alpha}(s,\bx)=\inf_{\tau\in\T_{s,T}}J(s,\bx;\alpha,\tau)=\inf_{\tau\in\T_{0,T}}J(s,\bx;\alpha,\tau).
\]
Now, for any $(s,x)\in[0,T]\times\R^d$, take an arbitrary sequence $\{(s_n,x_n)\}_{n\in\N}\subset[0,T]\times\R^d$ such that $(s_n,x_n)\to (s,x)$. Then the continuity of $\widetilde{G}^\alpha(s,(x,1,0))$ can be seen from the following estimation
\begin{align*}
&\left|\widetilde{G}^\alpha(s_n,(x_n,1,0))-\widetilde{G}^\alpha(s,(x,1,0))\right| = \bigg|\inf_{\ms \tau\in\T_{0,T}}\E[F(\bX^{s_n,x_n,1,0,\alpha}_\tau)]-\inf_{\ms \tau\in\T_{0,T}}\E[F(\bX^{s,x,1,0,\alpha}_\tau)]\bigg|\\
&\hspace{0.2in}\le \sup_{\ms \tau\in\T_{0,T}}\E\left[\left|F(\bX^{s_n,x_n,1,0,\alpha}_\tau)-F(\bX^{s,x,1,0,\alpha}_\tau)\right|\right] \le\E\bigg[\sup_{\ms 0\le r\le T}\left|F(\bX^{s_n,x_n,1,0,\alpha}_r)-F(\bX^{s,x,1,0,\alpha}_r)\right|\bigg]\to 0,
\end{align*} 
where the convergence follows from Lemma~\ref{lem:convergence}.

{\bf (ii)} Suppose $\alpha\in\A_t$ for some $t\in[0,T]$. For any $(s,\bx)\in[0,t]\times\mathcal{S}$ and $\tau\in\mathcal{T}_{s,T}$, by taking $\theta=s$ in Lemma~\ref{lem:E = J}, we have
\begin{equation}\label{J > infJ}
J(s,\bx;\alpha,\tau)= \E\left[\E[F(\bX^{s,\bx,\alpha}_{\tau})\mid\F_s](\omega)\right] = \E\left[J(s,\bx;\alpha,\tau^{s,\omega})\right]
\ge\inf_{\tau\in\mathcal{T}^s_{s,T}}J(s,\bx;\alpha,\tau),
\end{equation}
where in the second equality we replace $\alpha^{s,\omega}$ by $\alpha$, thanks to Proposition~\ref{prop:alpha^t,w=alpha}. We then conclude
\begin{equation}\label{inf=inf}
\inf_{\tau\in\T_{s,T}}J(s,\bx;\alpha,\tau) = \inf_{\tau\in\T^s_{s,T}}J(s,\bx;\alpha,\tau),
\end{equation}
as the ``$\le$'' relation is trivial. That is, $\widetilde{G}^\alpha(s,\bx)=G^\alpha(s,\bx)$.
\end{proof}

Now, we want to modify the arguments in the proof of \cite[Theorem 3.5]{BT11} to get a weak dynamic programming principle for $V$. Given $w:[0,T]\times\R^d\mapsto\R$, we mimic the relation between $V$ and $\bar{V}$ in \eqref{V bV} and define $\bar{w}:[0,T]\times\mathcal{S}\mapsto\R$ by
\begin{equation}\label{def barphi}
\bar{w}(t,x,y,z):=z+y w(t,x),\ (t,x,y,z)\in[0,T]\times\mathcal{S}.
\end{equation}

\begin{prop}\label{prop:WDPP_super}
Fix $(t,\bx)\in[0,T]\times\mathcal{S}$ and $\varepsilon > 0$. Take arbitrary $\alpha\in\mathcal{A}_t$,  $\theta\in\mathcal{T}^t_{t,T}$ and $\varphi\in \operatorname{USC}([0,T]\times\mathbb{R}^d)$ with $\varphi\le V$. We have the following:
\begin{itemize}
\item[(i)] $\mathbb{E}[\bar{\varphi}^+(\theta,\bX^{t,\bx,\alpha}_\theta)]<\infty$;
\item[(ii)] If, moreover, $\mathbb{E}[\bar{\varphi}^-(\theta,\bX^{t,\bx,\alpha}_\theta)]<\infty$, then there exists $\alpha^*\in\mathcal{A}_t$ with $\alpha^*_s=\alpha_s$ for $s\in[0,\theta)$ such that  
\begin{equation*}
\mathbb{E}[F(\bX^{t,\bx,\alpha^*}_\tau)] \ge \mathbb{E}[Y^{t,x,y,\alpha}_{\tau\wedge\theta}\varphi(\tau\wedge\theta,X^{t,x,\alpha}_{\tau\wedge\theta})+Z^{t,x,y,z,\alpha}_{\tau\wedge\theta}]-4\varepsilon,\ \ \forall \tau\in\mathcal{T}^t_{t,T}.
\end{equation*}
\end{itemize}
\end{prop}

\begin{proof}
(i) First, observe that for any $\bx=(x,y,z)\in\mathcal{S}$, $\bar{\varphi}(t,\bx) = y\varphi(t,x)+z \le yV(t,x)+z \le yg(x)+z$, which implies $\bar{\varphi}^+(t,\bx)\le yg(x)+z$. It follows that 
\begin{eqnarray*}
\bar{\varphi}^+(\theta,\bX^{t,\bx,\alpha}_\theta) &\le& Y^{t,x,y,\alpha}_{\theta}g(X^{t,x,\alpha}_{\theta})+Z^{t,x,y,z,\alpha}_{\theta} \\
& \le &
Y^{t,x,y,\alpha}_{\theta}g(X^{t,x,\alpha}_{\theta}) + z + \int_{t}^{\theta}Y^{t,x,y,\alpha}_s f(s,X_{s}^{t,x,\alpha},\alpha_s)ds, 
\end{eqnarray*}
the right-hand-side is integrable as a result of \eqref{dominated}.

(ii)
For each $(s,\eta)\in[0,T]\times\mathcal{S}$, by the definition of $\bar{V}$, there exists $\alpha^{(s,\eta),\varepsilon}\in\mathcal{A}_s$ such that 
\begin{equation}\label{inf>V}
\inf_{\tau\in\mathcal{T}^s_{s,T}}J(s,\eta;\alpha^{(s,\eta),\varepsilon},\tau) \ge \bar{V}(s,\eta)-\varepsilon.
\end{equation}  
Note that $\varphi\in \operatorname{USC}([0,T]\times\mathbb{R}^d)$ implies $\bar{\varphi}\in \operatorname{USC}([0,T]\times\mathcal{S})$. Then by the upper semicontinuity of $\bar{\varphi}$ on $[0,T]\times\mathcal{S}$ and the lower semicontinuity of $G^{\alpha^{(s,\eta),\varepsilon}}$ on $[0,s]\times\mathcal{S}$ (from Lemma~\ref{lem:conti} (ii)), there must exist $r^{(s,\eta)}>0$ such that
\[
\bar{\varphi}(t',x')-\bar{\varphi}(s,\eta) \le \varepsilon\ \text{and}\ G^{\alpha^{(s,\eta),\varepsilon}}(s,\eta)-G^{\alpha^{(s,\eta),\varepsilon}}(t',x') \le \varepsilon\ \hbox{for all}\ (t',x')\in B(s,\eta;r^{(s,\eta)}),
\]
where $B(s,\eta;r)=\{(t',x')\in[0,T]\times\mathcal{S}\ |\ t'\in(s-r,s], |x'-\eta|<r\}$, defined as in \eqref{defn ball}. It follows that if $(t',x')\in B(s,\eta;r^{(s,\eta)})$, we have
\[
G^{\alpha^{(s,\eta),\varepsilon}}(t',x') \ge G^{\alpha^{(s,\eta),\varepsilon}}(s,\eta)-\varepsilon \ge \bar{V}(s,\eta)-2\varepsilon \ge \bar{\varphi}(s,\eta)-2\varepsilon \ge \bar{\varphi}(t',x')-3\varepsilon,
\]
where the second inequality is due to \eqref{inf>V}. Here, we do not use the usual topology induced by balls of the form $B_{r}(s,\eta)=\{(t',x')\in[0,T]\times\mathcal{S}\ |\ |t'-s|<r, |x'-\eta|<r\}$; instead, for the time variable, we consider the topology induced by half-closed intervals on $[0,T]$, i.e. the so-called upper limit topology (see e.g. \cite[Ex.4 on p.66]{Dugundji-book-66}). Note from \cite[Ex.3 on p.174]{Dugundji-book-66} and \cite[Ex.3 on p.192]{Munkres-book-75} that $(0,T]$ is a Lindel\"{o}f space under this topology. It follows that, under this setting, $\{B(s,\eta;r)\ |\ (s,\eta)\in[0,T]\times\mathcal{S}, 0< r\le r^{(s,\eta)}\}$ forms an open covering of $(0,T]\times\mathcal{S}$, and there exists a countable subcovering $\{B(t_i,x_i;r_i)\}_{i\in\mathbb{N}}$ of $(0,T]\times\mathcal{S}$. Now set $A_0 := \{T\}\times\mathcal{S}$, $C_{-1}:=\emptyset$ and define for all $i\in\mathbb{N}\cup\{0\}$
\[
A_{i+1}:=B(t_{i+1},x_{i+1};r_{i+1})\setminus C_i,\ \text{where}\ C_i:=C_{i-1}\cup A_i.
\]
Under this construction, we have 
\begin{equation}\label{Ai structure}
(\theta,\bX^{t,\bx,\alpha}_\theta)\in\cup_{i\in\mathbb{N}\cup\{0\}}A_i\ \P\hbox{-a.s.},\ A_i\cap A_j=\emptyset\ \text{for}\ i\neq j,\ \text{and}\  G^{\alpha^{i,\varepsilon}}(t',x') \ge \bar{\varphi}(t',x')-3\varepsilon\ \text{for}\ (t',x')\in A_i,
\end{equation}
where $\alpha^{i,\varepsilon}:=\alpha^{(t_i,x_i),\varepsilon}$.

For any $n\in\mathbb{N}$, set $A^n:=\cup_{0\le i\le n}A_i$ and define
\[
\alpha^{\eps, n} := \alpha1_{[0,\theta)} + \left(\alpha 1_{(A^n)^c}(\theta,\bX^{t,\bx,\alpha}_\theta)+\sum_{i=0}^{n}\alpha^{i,\eps}1_{A_i}(\theta,\bX^{t,\bx,\alpha}_\theta)\right)1_{[\theta,T]}\in\mathcal{A}_t.
\]
Note that $\alpha^{\varepsilon, n}_s=\alpha_s$ for $s\in[0,\theta)$. Whenever $\omega\in \{(\theta,\bX^{t,\bx,\alpha}_\theta)\in A_i\}$, observe that 
$(\alpha^{\eps,n})^{\theta,\omega}(\omega')=\alpha^{\eps,n}\left(\omega\otimes_\theta\phi_\theta(\omega')\right)=\alpha^{i,\eps}\left(\omega\otimes_\theta\phi_\theta(\omega')\right)=(\alpha^{i,\eps})^{\theta,\omega}(\omega')$; also, we have $\alpha^{i,\eps}\in\A_{\theta(\omega)}$, as $\alpha^{i,\eps}\in\A_{t_i}$ and $\theta(\omega)\le t_i$ on $A_i$. We then deduce from Lemma~\ref{lem:E = J}, Proposition~\ref{prop:alpha^t,w=alpha}, and \eqref{Ai structure} that for $\overline{\P}$-a.e. $\omega\in\Omega$
\begin{equation}\label{>phi}
\begin{split}
\mathbb{E}[F(\bX^{t,\bx,\alpha^{\varepsilon,n}}_\tau)1_{\{\tau\ge\theta\}}|\mathcal{F}_\theta]1_{A^n}(\theta,\bX^{t,\bx,\alpha}_\theta) &= 
1_{\{\tau\ge\theta\}}\sum_{i=0}^n J(\theta,\bX^{t,\bx,\alpha}_\theta;\alpha^{i,\varepsilon},\tau^{\theta,\omega})1_{A_i}(\theta,\bX^{t,\bx,\alpha}_\theta)\\
&\ge
1_{\{\tau\ge\theta\}}\sum_{i=0}^n G^{\alpha^{i,\varepsilon}}(\theta,\bX^{t,\bx,\alpha}_\theta)1_{A_i}(\theta,\bX^{t,\bx,\alpha}_\theta)\\
&\ge
1_{\{\tau\ge\theta\}} [\bar{\varphi}(\theta,\bX^{t,\bx,\alpha}_\theta)-3\varepsilon]1_{A^n}(\theta,\bX^{t,\bx,\alpha}_\theta).
\end{split}
\end{equation}
Hence, we have
\begin{equation}\label{>phi 2}
\begin{split}
\mathbb{E}[F(\bX^{t,\bx,\alpha^{\varepsilon,n}}_\tau)] &=
\mathbb{E}[F(\bX^{t,\bx,\alpha}_\tau)1_{\{\tau<\theta\}}]+\mathbb{E}[F(\bX^{t,\bx,\alpha^{\varepsilon,n}}_\tau)1_{\{\tau\ge\theta\}}]\\
&=
\mathbb{E}[F(\bX^{t,\bx,\alpha}_\tau)1_{\{\tau<\theta\}}]+\mathbb{E}\left[\mathbb{E}[F(\bX^{t,\bx,\alpha^{\varepsilon,n}}_\tau)1_{\{\tau\ge\theta\}}|\mathcal{F}_\theta]1_{A^n}(\theta,\bX^{t,\bx,\alpha}_\theta)\right]\\
&\hspace{0.2 in} +  
\mathbb{E}\left[\mathbb{E}[F(\bX^{t,\bx,\alpha^{\varepsilon,n}}_\tau)1_{\{\tau\ge\theta\}}|\mathcal{F}_\theta]1_{(A^n)^c}(\theta,\bX^{t,\bx,\alpha}_\theta)\right]\\
&\ge
\mathbb{E}[F(\bX^{t,\bx,\alpha}_\tau)1_{\{\tau<\theta\}}]+\mathbb{E}[1_{\{\tau\ge\theta\}}\bar{\varphi}(\theta,\bX^{t,\bx,\alpha}_\theta)1_{A^n}(\theta,\bX^{t,\bx,\alpha}_\theta)]-3\varepsilon\\
&\ge
\mathbb{E}[1_{\{\tau<\theta\}}\bar{\varphi}(\tau,\bX^{t,\bx,\alpha}_\tau)]+\mathbb{E}[1_{\{\tau\ge\theta\}}\bar{\varphi}(\theta,\bX^{t,\bx,\alpha}_\theta)1_{A^n}(\theta,\bX^{t,\bx,\alpha}_\theta)]-3\varepsilon,
\end{split}
\end{equation}
where the first inequality comes from \eqref{>phi}, and the second inequality is due to the observation that
\begin{eqnarray*}
F(\bX^{t,\bx,\alpha}_\tau) &=& Y^{t,x,y,\alpha}_\tau g(X^{t,x,\alpha}_\tau) + Z^{t,x,y,z,\alpha}_\tau \ge Y^{t,x,y,\alpha}_\tau V(\tau, X^{t,x,\alpha}_\tau) + Z^{t,x,y,z,\alpha}_\tau \\
&\ge&
Y^{t,x,y,\alpha}_\tau \varphi(\tau, X^{t,x,\alpha}_\tau) + Z^{t,x,y,z,\alpha}_\tau.
\end{eqnarray*}
Since $\mathbb{E}[\bar{\varphi}^+(\theta,\bX^{t,\bx,\alpha}_\theta)]<\infty$ (by part (i)), there exists $n^*\in\mathbb{N}$ such that  
\[
\mathbb{E}[\bar{\varphi}^+(\theta,\bX^{t,\bx,\alpha}_\theta)] - \mathbb{E}[\bar{\varphi}^+(\theta,\bX^{t,\bx,\alpha}_\theta)1_{A^{n^*}}(\theta,\bX^{t,\bx,\alpha}_\theta)] < \varepsilon.
\]
We observe the following holds for any $\tau\in\mathcal{T}^t_{t,T}$
\begin{equation}\label{<epsilon}
\begin{split}
&\mathbb{E}[1_{\{\tau\ge\theta\}}\bar{\varphi}^+(\theta,\bX^{t,\bx,\alpha}_\theta)] - \mathbb{E}[1_{\{\tau\ge\theta\}}\bar{\varphi}^+(\theta,\bX^{t,\bx,\alpha}_\theta)1_{A^{n^*}}(\theta,\bX^{t,\bx,\alpha}_\theta)]\\
&\hspace{2in} \leq
\mathbb{E}[\bar{\varphi}^+(\theta,\bX^{t,\bx,\alpha}_\theta)] - \mathbb{E}[\bar{\varphi}^+(\theta,\bX^{t,\bx,\alpha}_\theta)1_{A^{n^*}}(\theta,\bX^{t,\bx,\alpha}_\theta)] < \varepsilon.
\end{split}
\end{equation}
Suppose $\mathbb{E}[\bar{\varphi}^-(\theta,\bX^{t,\bx,\alpha}_\theta)]<\infty$, then we can conclude from \eqref{<epsilon} that for any $\tau\in\mathcal{T}^t_{t,T}$ 
\begin{align}\label{<epsilon 2}
&\mathbb{E}[1_{\{\tau\ge\theta\}}\bar{\varphi}(\theta,\bX^{t,\bx,\alpha}_\theta)] 
= \mathbb{E}[1_{\{\tau\ge\theta\}}\bar{\varphi}^+(\theta,\bX^{t,\bx,\alpha}_\theta)]-\mathbb{E}[1_{\{\tau\ge\theta\}}\bar{\varphi}^-(\theta,\bX^{t,\bx,\alpha}_\theta)]\nonumber\\
&\hspace{0.2in} \le 
\mathbb{E}[1_{\{\tau\ge\theta\}}\bar{\varphi}^+(\theta,\bX^{t,\bx,\alpha}_\theta)1_{A^{n^*}}(\theta,\bX^{t,\bx,\alpha}_\theta)]+\varepsilon-\mathbb{E}[1_{\{\tau\ge\theta\}}\bar{\varphi}^-(\theta,\bX^{t,\bx,\alpha}_\theta)1_{A^{n^*}}(\theta,\bX^{t,\bx,\alpha}_\theta)]\nonumber\\
&\hspace{0.2in} = 
\mathbb{E}[1_{\{\tau\ge\theta\}}\bar{\varphi}(\theta,\bX^{t,\bx,\alpha}_\theta)1_{A^{n^*}}(\theta,\bX^{t,\bx,\alpha}_\theta)]+\varepsilon.
\end{align}
Taking $\alpha^*=\alpha^{\varepsilon,n^*}$, we now conclude from \eqref{>phi 2} and \eqref{<epsilon 2} that
\begin{align*}\label{>-4epsilon}
\mathbb{E}[F(\bX^{t,\bx,\alpha^*}_\tau)] &\ge \mathbb{E}[1_{\{\tau<\theta\}}\bar{\varphi}(\tau,\bX^{t,\bx,\alpha}_\tau)]+\mathbb{E}[1_{\{\tau\ge\theta\}}\bar{\varphi}(\theta,\bX^{t,\bx,\alpha}_\theta)]-4\varepsilon\\
&=
\mathbb{E}[\bar{\varphi}(\tau\wedge\theta,\bX^{t,\bx,\alpha}_{\tau\wedge\theta})]-4\eps\\
&=E[Y^{t,x,y,\alpha}_{\tau\wedge\theta}\varphi(\tau\wedge\theta,X^{t,x,\alpha}_{\tau\wedge\theta})+Z^{t,x,y,z,\alpha}_{\tau\wedge\theta}]-4\varepsilon.
\end{align*}
\end{proof}

We still need the following property of $V$ in order to obtain the supersolution property.

\begin{prop}\label{prop:V=sup tG}
For any $(t,x)\in[0,T]\times\R^d$, $V(t,x)=\sup_{\alpha\in\A}\widetilde{G}^\alpha(t,(x,1,0))$.
\end{prop}

\begin{proof}
Thanks to Lemma~\ref{lem:conti} (ii), we immediately have
\[
V(t,x)=\sup_{\alpha\in\A_t}{G}^\alpha(t,(x,1,0))=\sup_{\alpha\in\A_t}\widetilde{G}^\alpha(t,(x,1,0))\le\sup_{\alpha\in\A}\widetilde{G}^\alpha(t,(x,1,0)).
\]
For the reverse inequality, fix $\alpha\in\A$ and $\bx\in\mathcal{S}$. By a calculation similar to \eqref{J > infJ}, we have $J(t,\bx;\alpha,\tau)=\E[J(t,\bx;\alpha^{t,\omega},\tau^{t,\omega})]$, for any $\tau\in\T_{t,T}$. Observing that $\tau^{t,\omega}\in\T^t_{t,T}$ for all $\tau\in\T_{t,T}$ (by Proposition~\ref{lem: tau a stopping time a.s.}), and that $\E[J(t,\bx;\alpha^{t,\omega},\tau^{t,\omega})]=\E[J(t,\bx;\alpha^{t,\omega},\tau)]$ for all $\tau\in\T^t_{t,T}$ (by Proposition~\ref{prop:xi indep. of F_t}), we obtain
\begin{eqnarray*}
\inf_{\tau\in\T_{t,T}}J(t,\bx;\alpha,\tau)&=&\inf_{\tau\in\T_{t,T}}\E[J(t,\bx;\alpha^{t,\omega},\tau^{t,\omega})]=\inf_{\tau\in\T^t_{t,T}}\E[J(t,\bx;\alpha^{t,\omega},\tau)]\\
&\le& \sup_{\alpha\in\A_t}\inf_{\tau\in\T^t_{t,T}}\E[J(t,\bx;\alpha,\tau)]=\sup_{\alpha\in\A_t}\inf_{\tau\in\T^t_{t,T}}J(t,\bx;\alpha,\tau),
\end{eqnarray*}  
where the inequality is due to the fact that $\alpha^{t,\omega}\in\A_t$. By setting $\bx:=(x,1,0)$ and taking supremum over $\alpha\in\A$, we get $\sup_{\alpha\in\A}\widetilde{G}^\alpha(t,(x,1,0))\le V(t,x)$.
\end{proof}

\begin{coro}\label{coro:V LSC}
$V\in \operatorname{LSC}([0,T]\times\R^d)$.
\end{coro}

\begin{proof}
By Proposition~\ref{prop:V=sup tG} and Lemma~\ref{lem:conti} (i), $V$ is a supremum of a collection of continuous functions defined on $[0,T]\times\R^d$, and thus has to be lower semicontinuous on the same space.
\end{proof}

Now, we are ready to present the main result of this section. Recall that the operator $\overline{H}$ is defined in \eqref{bH}.

\begin{prop}\label{prop:vis super}
The function $V$ is a lower semicontinuous viscosity supersolution to the obstacle problem of a Hamilton-Jacobi-Bellman equation 
\begin{equation}\label{HJB obstacle 1}
\max\left\{ c(t,x)w-\frac{\partial w}{\partial t}+\overline{H}(t,x,D_x w, D^2_x w),\ w-g(x)\right\}=0\ \hbox{on}\ [0,T)\times\R^d,
\end{equation}
and satisfies the polynomial growth condition: there exists $N>0$ such that
\begin{equation}\label{poly grow U,V}
|V(t,x)|\le N(1+|x|^{\bar{p}}),\ \forall (t,x)\in[0,T]\times\R^d.
\end{equation}
\end{prop}

\begin{proof}
The lower semicontinuity of $V$ was shown in Corollary~\ref{coro:V LSC}. Observe that $0\le V(t,x)\le \sup_{\alpha\in\A_t}\E[F(\bX^{t,x,1,0,\alpha}_T)]\le \sup_{\alpha\in\A}\E[F(\bX^{t,x,1,0,\alpha}_T)]=:v(t,x)$. Since $v$ satisfies \eqref{poly grow U,V} as a result of \cite[Theorem 3.1.5]{Krylov-book-80}, so does $V$. 
 
To prove the supersolution property, let $h\in C^{1,2}([0,T)\times\mathbb{R}^d)$ be such that
\begin{equation}\label{min of V-h}
0 = (V-h)(t_0,x_0) < (V-h)(t,x),\ \text{for any}\ (t,x)\in[0,T)\times\mathbb{R}^d,\ (t,x)\neq(t_0,x_0),
\end{equation}
for some $(t_0,x_0)\in[0,T)\times\mathbb{R}^d$. If $V(t_0,x_0)=g(x_0)$, then there is nothing to prove. We, therefore, assume that $V(t_0, x_0)<g(x_0)$. For such $(t_0,x_0)$ it is enough to prove the following inequality:
\begin{equation*}
0\leq c(t_0,x_0)h(t_0,x_0)-\frac{\partial h}{\partial t}(t_0,x_0)+\overline{H}(\cdot,D_x h,D^2_x h)(t_0,x_0).
\end{equation*} 
Assume the contrary. Then, by the definition of $\overline{H}$ in \eqref{bH}, there must exist $\zeta_0\in M$ such that 
\begin{equation*}\label{0>}
0 > c(t_0,x_0)h(t_0,x_0)-\frac{\partial h}{\partial t}(t_0,x_0)+(H^{\zeta_0})^*(\cdot,D_x h,D^2_x h)(t_0,x_0).
\end{equation*} 
Moreover, from the upper semicontinuity of $(H^{\zeta_0})^*$ and the fact that $(H^{\zeta_0})^*\ge H^{\zeta_0}$, we can choose some $r>0$ with $t_0+r<T$ such that
\begin{equation}\label{submart.}
0 > c(t,x)h(t,x)-\frac{\partial h}{\partial t}(t,x)+H^{\zeta_0}(\cdot,D_x h,D^2_x h)(t,x),\ \text{for all}\ (t,x)\in \bar{B}_{r}(t_0,x_0). 
\end{equation}
Define $\zeta\in\mathcal{A}$ by setting $\zeta_t=\zeta_0$ for all $t\geq 0$, and introduce the stopping time 
\[
\theta:=\inf\left\{s\ge t_0\ \middle|\ (s,X_s^{t_0,x_0,\zeta})\notin B_{r}(t_0,x_0)\right\}\in\mathcal{T}^{t_0}_{t_0,T}.
\]
Note that we have $\theta\in\mathcal{T}^{t_0}_{t_0,T}$ as the control $\zeta$ is by definition independent of $\mathcal{F}_{t_0}$. Now, by applying the product rule of stochastic calculus to $Y^{t_0,x_0,1,\zeta}_s h(s,X^{t_0,x_0,\zeta}_s)$ and recalling \eqref{submart.} and $c\le \bar{c}$, we obtain that for any $\tau\in\T_{t_0,T}^{t_0}$,
\begin{align}\label{bdd below}
V(t_0,x_0)=h(t_0,x_0) &= \E\bigg[Y^{t_0,x_0,1,\zeta}_{\theta\wedge\tau} h(\theta\wedge\tau,X^{t_0,x_0,\zeta}_{\theta\wedge\tau})\nonumber \\&\hspace{0.4in}+ \int_{t_0}^{\theta\wedge\tau} Y^{t_0,x_0,1,\zeta}_s\left(c h-\frac{\partial h}{\partial t}+H^{\zeta_0}(\cdot,D_x h, D_{x}^2 h)+f\right)(s,X^{t_0,x_0,\zeta}_{s},\zeta_0) ds \bigg]\nonumber\\
& < 
\E\left[Y^{t_0,x_0,1,\zeta}_{\theta\wedge\tau} h(\theta\wedge\tau,X^{t_0,x_0,\zeta}_{\theta\wedge\tau}) + \int_{t_0}^{\theta\wedge\tau} Y^{t_0,x_0,1,\zeta}_s f(s,X^{t_0,x_0,\zeta}_{s},\zeta_0) ds \right].
\end{align}
In the following, we will work towards a contradiction to \eqref{bdd below}. First, define
\[
\bar{h}(\theta,\bX^{t_0,x_0,1,0,\zeta}_{\theta}):=Y^{t_0,x_0,1,\zeta}_{\theta} h(\theta,X^{t_0,x_0,\zeta}_{\theta}) + \int_{t_0}^{\theta} Y^{t_0,x_0,1,\zeta}_s f(s,X^{t_0,x_0,\zeta}_{s},\zeta_0) ds.
\]
Note from \eqref{bdd below} that $\E[\bar{h}(\theta,\bX^{t_0,x_0,1,0,\zeta}_{\theta})]$ is bounded from below.  It follows from this fact that $\E[\bar{h}^-(\theta,\bX^{t_0,x_0,1,0,\zeta}_{\theta})]<\infty$, as we already have $\E[\bar{h}^+(\theta,\bX^{t_0,x_0,1,0,\zeta}_{\theta})]<\infty$ from Proposition~\ref{prop:WDPP_super} (i). For each $n\in\N$, we can therefore apply Proposition~\ref{prop:WDPP_super} (ii) and conclude that there exists $\alpha^{*,n}\in\mathcal{A}_{t_{0}}$, with $\alpha^{*,n}_s=\zeta_s$ for all $s\le \theta$, such that for any $\tau\in\mathcal{T}^{t_{0}}_{t_0,T}$,
\begin{equation}\label{apply Prop2.1}
\E[F(\bX^{t_0,x_0,1,0,\alpha^{*,n}}_\tau)] \ge \E\left[Y^{t_0,x_0,1,\zeta}_{\theta\wedge\tau} h(\theta\wedge\tau, X^{t_0,x_0,\zeta}_{\theta\wedge\tau}) + \int_{t_0}^{\theta\wedge\tau} Y^{t_0,x_0,1,\zeta}_s f(s,X^{t_0,x_0,\zeta}_{s},\zeta_0) ds \right]-\frac{1}{n}.
\end{equation}
Next, thanks to the definition of $V$ and the classical theory of Snell envelopes (see e.g. Appendix D, and especially Theorem D.12, in \cite{KS-book-98}), we have
\begin{equation}\label{V=hattau}
V(t_0,x_0)\ge G^{\alpha^{*,n}}(t_0,(x_0,1,0))= 
\E[F(\bX^{t_0,x_0,1,0,\alpha^{*,n}}_{\tau^n})],
\end{equation}
where 
\[
\tau^n:=\inf\left\{s\ge t_{0}\ \middle|\ G^{\alpha^{*,n}}(s,\bX^{t_0,x_0,1,0,\alpha^{*,n}}_s)=g(X^{t_0,x_0,\alpha^{*,n}}_s)\right\}\in\T^{t_0}_{t_0,T}.
\]
Note that we may apply \cite[Theorem D.12]{KS-book-98} because \eqref{dominated} holds. Combining \eqref{V=hattau} and \eqref{apply Prop2.1}, we obtain
\[
V(t_0,x_0)\ge \E\left[Y^{t_0,x_0,1,\zeta}_{\theta\wedge\tau^n} h(\theta\wedge\tau^n, X^{t_0,x_0,\zeta}_{\theta\wedge\tau^n}) + \int_{t_0}^{\theta\wedge\tau^n} Y^{t_0,x_0,1,\zeta}_s f(s,X^{t_0,x_0,\zeta}_{s},\zeta_0) ds \right]-\frac{1}{n}.
\]
By sending $n$ to infinity and using Fatou's Lemma, we conclude that 
\[
V(t_0,x_0)\ge \E\bigg[Y^{t_0,x_0,1,\zeta}_{\theta\wedge\tau^*} h(\theta\wedge\tau^*, X^{t_0,x_0,\zeta}_{\theta\wedge\tau^*}) + \int_{t_0}^{\theta\wedge\tau^*} Y^{t_0,x_0,1,\zeta}_s f(s,X^{t_0,x_0,\zeta}_{s},\zeta_0) ds \bigg],
\]
where $\tau^*:=\liminf_{n\to\infty}\tau^n$ is a stopping time in $\T^{t_0}_{t_0,T}$, thanks to the right continuity of the filtration $\mathbb{F}^{t_0}$. The above inequality, however, contradicts \eqref{bdd below}. 
\end{proof}

\begin{rem}
The lower semicontinuity of $V$ is needed for the proof of Proposition~\ref{prop:vis super}. To see this, suppose $V$ is not lower semicontinuous. Then $V$ should be replaced by $V_*$ in \eqref{min of V-h} and \eqref{bdd below}. The last inequality in the proof and \eqref{bdd below} would then yield $V_*(t_0,x_0)<V(t_0,x_0)$, which is not a contradiction.   
\end{rem}

\begin{rem}\label{rem:why bH}
Due to the lack of continuity in $t$ of the functions $b$, $\sigma$, and $f$, we use $\overline{H}$, instead of $H^*$, in \eqref{HJB obstacle 1}. If we were using $H^*$, we in general would not be able to find a $\zeta_0\in M$ such that \eqref{submart.} holds (due to the lack of continuity in 
$t$). If $b$, $\sigma$, and $f$ are actually continuous in $t$, then we see from \eqref{H} and \eqref{bH} that $\overline{H}=H=H^*$.  
\end{rem}


\section{Subsolution Property of $U^*$}\label{sec:sub}
As in Section~\ref{sec:super}, we will first prove a continuity result (Lemma~\ref{lem:L continuous}), which leads to a weak dynamic programming principle for $U$ (Proposition~\ref{prop:WDPP_sub}). Then, we will show that the subsolution property of $U^*$ follows from this weak dynamic programming principle (Proposition~\ref{prop:vis sub}). Remember that $U^*$ is the upper semicontinuous envelope of $U$ defined as in \eqref{envelopes}.

Fix $s\in[0,T]$ and $\xi\in L^p_d(\Omega,\F_s)$ for some $p\in[1,\infty)$. For any $\alpha\in\A$ and $\pi_1,\pi_2\in\Pi_{s,T}$ with $\pi_1[\beta]\le\pi_2[\beta]$ $\overline{\P}$-a.s. for all $\beta\in\A$, we define 
\begin{equation}\label{defn B}
\B^{s,\xi,\alpha}_{\pi_1}:=\bigg\{\beta\in\A\ \bigg|\ \int_s^{\pi_1[\alpha]}\rho'(\beta_u,\alpha_u)du=0\ \ \overline{\P}\hbox{-a.s.}\bigg\},
\end{equation}
and introduce the random variable
\begin{equation}\label{defn K}
\begin{split} 
K^{s,\xi,\alpha}(\pi_1,\pi_2)&:=\\
\esssup_{\beta\in\B^{s,\xi,\alpha}_{\pi_1}}\E&\left[\int_{\pi_1[\alpha]}^{\pi_2[\beta]}Y^{\pi_1[\alpha],X^{s,\xi,\beta}_{\pi_1[\alpha]},1,\beta}_u f(u,X^{s,\xi,\beta}_u,\beta_u)du + Y^{\pi_1[\alpha],X^{s,\xi,\beta}_{\pi_1[\alpha]},1,\beta}_{\pi_2[\beta]}g(X^{s,\xi,\beta}_{\pi_2[\beta]})\ \middle|\ \F_{\pi_1[\alpha]}\right].
\end{split}
\end{equation}
Observe from the definition of $\B^{s,\xi,\alpha}_{\pi_1}$ and Definition~\ref{defn:strategy} (i) that 
\begin{equation}\label{pi[beta]=pi[alpha]}
\pi_1[\beta]=\pi_1[\alpha]\ \ \overline{\P}\hbox{-a.s.}\ \ \forall \beta\in\B^{s,\xi,\alpha}_{\pi_1}.
\end{equation}
This in particular implies $\pi_2[\beta]\ge\pi_1[\beta]=\pi_1[\alpha]$ $\overline{\P}$-a.s. $\forall \beta\in\B^{s,\xi,\alpha}_{\pi_1}$, which shows that $K^{s,\xi,\alpha}(\pi_1,\pi_2)$ is well-defined. Given any constant strategies $\pi_1[\cdot]\equiv\tau_1\in\T^s_{s,T}$ and $\pi_2[\cdot]\equiv\tau_2\in\T^s_{s,T}$, we will simply write $K^{s,\xi,\alpha}(\pi_1,\pi_2)$ as $K^{s,\xi,\alpha}(\tau_1,\tau_2)$. For the particular case where $\xi=x\in\R^d$, we also consider
\begin{equation*}
\Gamma^{s,x,\alpha}(\pi_1,\pi_2):=\int_s^{\pi_1[\alpha]}Y^{s,x,1,\alpha}_u f(u,X^{s,x,\alpha}_u,\alpha_u)du + Y^{s,x,1,\alpha}_{\pi_1[\alpha]}K^{s,x,\alpha}(\pi_1,\pi_2).
\end{equation*}

\begin{rem}\label{rem:esssup=lim}
Let us write $K^{s,x,\alpha}(\pi_1,\pi_2)=\esssup_{\beta\in\B^{s,x,\alpha}_{\pi_1}}\E[R^{s,x,\alpha}_{\pi_1,\pi_2}(\beta)\mid\F_{\pi_1[\alpha]}]$ for simplicity. Note that the set of random variables $\{\E[R^{s,x,\alpha}_{\pi_1,\pi_2}(\beta)\mid\F_{\pi_1[\alpha]}]\}_{\beta\in\B^{s,x,\alpha}_{\pi_1}}$ is closed under pairwise maximization. Indeed, given $\beta_1,\beta_2\in\B^{s,x,\alpha}_{\pi_1}$, set $A:=\{\E[R^{s,x,\alpha}_{\pi_1,\pi_2}(\beta_1)\mid\F_{\pi_1[\alpha]}]\ge \E[R^{s,x,\alpha}_{\pi_1,\pi_2}(\beta_2)\mid\F_{\pi_1[\alpha]}]\}\in\F_{\pi_1[\alpha]}$ and define $\beta_3:=\beta_1 1_{[0,\pi_1[\alpha])}+(\beta_1 1_A +\beta_2 1_{A^c})1_{[\pi_1[\alpha],T]}\in\B^{s,x,\alpha}_{\pi_1}$. Then, observe that
\begin{equation*}
\begin{split}
\E[R^{s,x,\alpha}_{\pi_1,\pi_2}(\beta_3)\mid\F_{\pi_1[\alpha]}]&=\E[R^{s,x,\alpha}_{\pi_1,\pi_2}(\beta_1)\mid\F_{\pi_1[\alpha]}]1_A+\E[R^{s,x,\alpha}_{\pi_1,\pi_2}(\beta_2)\mid\F_{\pi_1[\alpha]}]1_{A^c}\\
&=\E[R^{s,x,\alpha}_{\pi_1,\pi_2}(\beta_1)\mid\F_{\pi_1[\alpha]}]\vee\E[R^{s,x,\alpha}_{\pi_1,\pi_2}(\beta_2)\mid\F_{\pi_1[\alpha]}].
\end{split}
\end{equation*}
Thus, we conclude from Theorem A.3 in \cite[Appendix A]{KS-book-98} that there exists a sequence $\{\beta^n\}_{n\in\N}$ in $\B^{s,x,\alpha}_{\pi_1}$ such that $K^{s,x,\alpha}(\pi_1,\pi_2)=\uparrow\lim_{n\to\infty}\E[R^{s,x,\alpha}_{\pi_1,\pi_2}(\beta^n)\mid\F_{\pi_1[\alpha]}]$ $\overline{\P}$-a.s.
\end{rem}

\begin{lem}\label{lem:K^s=K^r}
Fix $(s,x)\in[0,T]\times\R^d$ and $\alpha\in\A$. For any $r\in[s,T]$ and $\pi\in\Pi_{r,T}$, we have
\[
K^{s,x,\alpha}(r,\pi)=K^{r,X^{s,x,\alpha}_r,\alpha}(r,\pi)\ \overline{\P}\hbox{-a.s.}
\]
\end{lem}

\begin{proof}
For any $\beta\in\B^{s,x,\alpha}_{r}$, we see from Remark~\ref{rem:flow property} (i) that $X^{s,x,\beta}_u=X^{r,X^{s,x,\alpha}_r,\beta}_u$ for $u\in[r,T]$ $\overline{\P}$-a.s. It follows from \eqref{defn K} that
\begin{equation*}
K^{s,x,\alpha}(r,\pi)=\esssup_{\beta\in\B^{s,x,\alpha}_r}\E\bigg[\int_{r}^{\pi[\beta]}Y^{r,X^{s,x,\alpha}_{r},1,\beta}_u f(u,X^{r,X^{s,x,\alpha}_r,\beta}_u,\beta_u)du + Y^{r,X^{s,x,\alpha}_{r},1,\beta}_{\pi[\beta]}g(X^{r,X^{s,x,\alpha}_r,\beta}_{\pi[\beta]})\ \bigg|\ \F_{r}\bigg].
\end{equation*}
Observing from \eqref{defn B} that $\B^{s,x,\alpha}_r\subseteq\A=\B^{r,X^{s,x,\alpha}_r,\alpha}_r$, we conclude $K^{s,x,\alpha}(r,\pi)\le K^{r,X^{s,x,\alpha}_r,\alpha}(r,\pi)$. On the other hand, for any $\beta\in\A$, define $\bar{\beta}:=\alpha 1_{[0,r)}+\beta 1_{[r,T]}\in\B^{s,x,\alpha}_r$. Then, by Remark~\ref{rem:flow property} (i) again, we have $X^{s,x,\bar{\beta}}_u=X^{r,X^{s,x,\alpha}_r,\beta}_u$ for $u\in[r,T]$ $\overline{\P}$-a.s. Also, we have $\pi[\bar{\beta}]=\pi[\beta]$, thanks to Definition~\ref{defn:strategy} (i). Therefore,
\begin{equation*}
\begin{split}
&\E\bigg[\int_{r}^{\pi[\beta]}Y^{r,X^{s,x,\alpha}_{r},1,\beta}_u f(u,X^{r,X^{s,x,\alpha}_r,\beta}_u,\beta_u)du + Y^{r,X^{s,x,\alpha}_{r},1,\beta}_{\pi[\beta]}g(X^{r,X^{s,x,\alpha}_r,\beta}_{\pi[\beta]})\ \bigg|\ \F_{r}\bigg]\\
&\hspace{0.7in}=\E\bigg[\int_{r}^{\pi[\bar{\beta}]}Y^{r,X^{s,x,\bar{\beta}}_{r},1,\bar{\beta}}_u f(u,X^{s,x,\bar{\beta}}_u,\bar{\beta}_u)du + Y^{r,X^{s,x,\bar{\beta}}_{r},1,\bar{\beta}}_{\pi[\bar{\beta}]}g(X^{s,x,\bar{\beta}}_{\pi[\bar{\beta}]})\ \bigg|\ \F_{r}\bigg].
\end{split}
\end{equation*} 
In view of \eqref{defn K}, this implies $K^{r,X^{s,x,\alpha}_r,\alpha}(r,\pi)\le K^{s,x,\alpha}(r,\pi)$.
\end{proof}

\begin{lem}\label{lem:supermartingale}
Fix $(s,x)\in[0,T]\times\R^d$. Given $\alpha\in\A$ and $\pi_1,\pi_2,\pi_3\in\Pi_{s,T}$ with $\pi_1[\beta]\le\pi_2[\beta]\le\pi_3[\beta]$ $\overline{\P}$-a.s. for all $\beta\in\A$, 
it holds $\overline{\P}$-a.s. that
\begin{equation*}
\E\bigg[\int_{\pi_1[\alpha]}^{\pi_2[\alpha]}Y^{s,x,1,\alpha}_u f(u,X^{s,x,\alpha}_u,\alpha_u)du + Y^{s,x,1,\alpha}_{\pi_2[\alpha]}K^{s,x,\alpha}(\pi_2,\pi_3)\ \bigg|\ \F_{\pi_1[\alpha]}\bigg]\le Y^{s,x,1,\alpha}_{\pi_1[\alpha]}K^{s,x,\alpha}(\pi_1,\pi_3).
\end{equation*}
Moreover, we have the following supermartingale property:
\begin{equation*}
\E[\Gamma^{s,x,\alpha}(\pi_2,\pi_3)\mid\F_{\pi_1[\alpha]}]\le\Gamma^{s,x,\alpha}(\pi_1,\pi_3)\ \overline{\P}\hbox{-a.s.}
\end{equation*}
\end{lem}

\begin{proof}
By Remark~\ref{rem:esssup=lim}, there exists a sequence $\{\beta^n\}_{n\in\N}$ in $\B^{s,x,\alpha}_{\pi_2}$ such that $K^{s,x,\alpha}(\pi_2,\pi_3)=\uparrow\lim_{n\to\infty}\E[R^{s,x,\alpha}_{\pi_2,\pi_3}(\beta^n)\mid\F_{\pi_2[\alpha]}]$ $\overline{\P}$-a.s. 
From the definition of $\B^{s,x,\alpha}_{\pi_2}$ in \eqref{defn B}, $\beta^n_u = \alpha_u$ for a.e. $u\in[s,\pi_2[\alpha])$ $\overline{\P}$-a.s. We can then compute as follows:
\begin{equation*}
\begin{split}
&\E\left[Y^{s,x,1,\alpha}_{\pi_2[\alpha]}K^{s,x,\alpha}(\pi_2,\pi_3)\ \middle|\ \F_{\pi_1[\alpha]}\right]\\
&=\E\bigg\{Y^{s,x,1,\alpha}_{\pi_2[\alpha]}\cdot\\
&\hspace{0.1in}\lim_{n\to\infty}\E\bigg[\int_{\pi_2[\alpha]}^{\pi_3[\beta^n]}Y^{\pi_2[\alpha],X^{s,x,\beta^n}_{\pi_2[\alpha]},1,\beta^n}_u f(u,X^{s,x,\beta^n}_u,\beta^n_u)du + Y^{\pi_2[\alpha],X^{s,x,\beta^n}_{\pi_2[\alpha]},1,\beta^n}_{\pi_3[\beta^n]}g(X^{s,x,\beta^n}_{\pi_3[\beta^n]})\bigg| \F_{\pi_2[\alpha]}\bigg]\bigg| \F_{\pi_1[\alpha]}\bigg\}\\
&=\E\bigg\{\lim_{n\to\infty}\E\bigg[\int_{\pi_2[\alpha]}^{\pi_3[\beta^n]}Y^{s,x,1,\beta^n}_u f(u,X^{s,x,\beta^n}_u,\beta^n_u)du + Y^{s,x,1,\beta^n}_{\pi_3[\beta^n]}g(X^{s,x,\beta^n}_{\pi_3[\beta^n]})\ \bigg|\ \F_{\pi_2[\alpha]}\bigg]\ \bigg|\ \F_{\pi_1[\alpha]} \bigg\}\\
&=\lim_{n\to\infty}\E\bigg[\int_{\pi_2[\alpha]}^{\pi_3[\beta^n]}Y^{s,x,1,\beta^n}_u f(u,X^{s,x,\beta^n}_u,\beta^n_u)du + Y^{s,x,1,\beta^n}_{\pi_3[\beta^n]}g(X^{s,x,\beta^n}_{\pi_3[\beta^n]})\ \bigg|\ \F_{\pi_1[\alpha]}\bigg],
\end{split}
\end{equation*}
where the last line follows from the monotone convergence theorem and the tower property for conditional expectations. We therefore conclude that
\begin{equation*}
\begin{split}
&\E\bigg[\int_{\pi_1[\alpha]}^{\pi_2[\alpha]}Y^{s,x,1,\alpha}_u f(u,X^{s,x,\alpha}_u,\alpha_u)du + Y^{s,x,1,\alpha}_{\pi_2[\alpha]}K^{s,x,\alpha}(\pi_2,\pi_3)\ \bigg|\ \F_{\pi_1[\alpha]}\bigg]\\
&=\lim_{n\to\infty}\E\bigg[\int_{\pi_1[\alpha]}^{\pi_3[\beta^n]}Y^{s,x,1,\beta^n}_u f(u,X^{s,x,\beta^n}_u,\beta^n_u)du + Y^{s,x,1,\beta^n}_{\pi_3[\beta^n]}g(X^{s,x,\beta^n}_{\pi_3[\beta^n]})\ \bigg|\ \F_{\pi_1[\alpha]}\bigg]\\
&=Y^{s,x,1,\alpha}_{\pi_1[\alpha]}\lim_{n\to\infty}\E\bigg[\int_{\pi_1[\alpha]}^{\pi_3[\beta^n]}Y^{\pi_1[\alpha],X^{s,x,\beta^n}_{\pi_1[\alpha]},1,\beta^n}_u f(u,X^{s,x,\beta^n}_u,\beta^n_u)du + Y^{\pi_1[\alpha],X^{s,x,\beta^n}_{\pi_1[\alpha]},1,\beta^n}_{\pi_3[\beta^n]}g(X^{s,x,\beta^n}_{\pi_3[\beta^n]}) \bigg| \F_{\pi_1[\alpha]}\bigg]\\
&\le Y^{s,x,1,\alpha}_{\pi_1[\alpha]} K^{s,x,\alpha}(\pi_1,\pi_3),
\end{split}
\end{equation*} 
where the inequality follows from the fact that $\beta^n\in\B^{s,x,\alpha}_{\pi_2}\subseteq\B^{s,x,\alpha}_{\pi_1}$. It then follows that
\begin{equation*}
\begin{split}
\E[\Gamma^{s,x,\alpha}(\pi_2,\pi_3)\mid\F_{\pi_1[\alpha]}]&=\int_{s}^{\pi_1[\alpha]}Y^{s,x,1,\alpha}_{u}f(u,X^{s,x,\alpha}_u,\alpha_u)du\\
&\hspace{0.2in}+\E\bigg[\int_{\pi_1[\alpha]}^{\pi_2[\alpha]}Y^{s,x,1,\alpha}_u f(u,X^{s,x,\alpha}_u,\alpha_u)du + Y^{s,x,1,\alpha}_{\pi_2[\alpha]}K^{s,x,\alpha}(\pi_2,\pi_3)\ \bigg|\ \F_{\pi_1[\alpha]}\bigg]\\
&\le \int_{s}^{\pi_1[\alpha]}Y^{s,x,1,\alpha}_{u}f(u,X^{s,x,\alpha}_u,\alpha_u)du+Y^{s,x,1,\alpha}_{\pi_1[\alpha]}K^{s,x,\alpha}(\pi_1,\pi_3)=\Gamma^{s,x,\alpha}(\pi_1,\pi_3).
\end{split}
\end{equation*}
\end{proof}

\begin{lem}\label{lem:A=A_t}
For any $(t,\bx)\in[0,T]\times\mathcal{S}$ and $\pi\in\Pi_{t,T}$,
\[
\sup_{\alpha\in\A}J(t,\bx;\alpha,\pi[\alpha])=\sup_{\alpha\in\A_t}J(t,\bx;\alpha,\pi[\alpha]).
\]
\end{lem}

\begin{proof}
Fix $\alpha\in\mathcal{A}$ and $\bx\in\mathcal{S}$. 
For any $\pi\in\Pi_{t,T}$, by taking $\theta=t$ in Lemma~\ref{lem:E = J}, we have
\begin{equation*}\label{J < supJ}
J(t,\bx;\alpha,\pi[\alpha])= \E\left[\E[F(\bX^{t,\bx,\alpha}_{\pi[\alpha]})\mid\F_t](\omega)\right] = \E\left[J(t,\bx;\alpha^{t,\omega},\pi[\alpha^{t,\omega}])\right]
\le\sup_{\alpha\in\mathcal{A}_t} J(t,\bx;\alpha,\pi[\alpha]).
\end{equation*}
Note that in the second equality we replace $\pi[\alpha]^{t,\omega}$ by $\pi[\alpha^{t,\omega}]$, thanks to Definition~\ref{defn:strategy} (iii). Then, the last inequality holds as $\alpha^{t,\omega}\in\A_t$ for $\overline{\P}$-a.e. $\omega\in\Omega$. Now, by taking supremum over $\alpha\in\A$, we have $\sup_{\alpha\in\A}J(t,\bx;\alpha,\pi[\alpha])\le\sup_{\alpha\in\A_t}J(t,\bx;\alpha,\pi[\alpha])$. Since the reverse inequality is trivial, this lemma follows.
\end{proof}

Now, we are ready to state a continuity result for an optimal control problem. 

\begin{lem}\label{lem:L continuous}
Fix $t\in[0,T]$. For any $\pi\in\Pi_{t,T}$, the function $L^\pi:[0,t]\times\mathcal{S}$ defined by
\begin{equation}\label{L}
L^\pi(s,\bx):=\sup_{\alpha\in\A_s}J(s,\bx;\alpha,\pi[\alpha])
\end{equation}
is continuous.
\end{lem}

\begin{proof}
Observing from \eqref{J = z+yJ} that $L^\pi(s,\bx)=yL^\pi(s,(x,1,0))+z$, it is enough to show the continuity of $L^\pi(s,(x,1,0))$ in $(s,x)$ on $[0,t]\times\R^d$.
By \cite[Theorem 3.2.2]{Krylov-book-80}, we know that $J(s,(x,1,0);\alpha,\tau)$ is continuous in $x$ uniformly with respect to $s\in[0,t]$, $\alpha\in\A$, and $\tau\in\T_{t,T}$. This shows that the map $(s,x,\alpha)\mapsto J(s,(x,1,0);\alpha,\pi[\alpha])$ is continuous in $x$ uniformly with respect to $s\in[0,t]$ and $\alpha\in\A$. Then, we see from the following estimation 
\[
\sup_{s\in[0,t]}|L^\pi(s,(x,1,0))-L^\pi(s,(x',1,0))|\le \sup_{s\in[0,t]}\sup_{\alpha\in\A_s}|J(s,(x,1,0);\alpha,\pi[\alpha])-J(s,(x',1,0);\alpha,\pi[\alpha])|
\]  
that $L^\pi(s,(x,1,0))$ is continuous in $x$ uniformly with respect to $s\in[0,t]$. Thus, it suffices to prove that $L^\pi(s,(x,1,0))$ is continuous in $s$ for each fixed $x$. To this end, we will first derive a dynamic programming principle for $L^\pi(s,(x,1,0))$, which corresponds to \cite[Theorem 3.3.6]{Krylov-book-80}; the rest of the proof will then follow from the same argument in \cite[Lemma 3.3.7]{Krylov-book-80}.  

Fix $(s,x)\in[0,t]\times\R^d$. Observe from \eqref{defn B} that $\B^{s,x,\alpha}_{s}=\A$ for all $\alpha\in\A$. In view of \eqref{defn K}, this implies that $K^{s,x,\alpha}(s,\pi)=\esssup_{\beta\in\A}\E[F(\bX^{s,x,1,0,\beta}_{\pi[\beta]})\mid\F_s]$, which is independent of $\alpha\in\A$. We will therefore drop the superscript $\alpha$ in the rest of the proof. Now, we claim that $K^{s,x}(s,\pi)$ is deterministic and equal to $L^\pi(s,(x,1,0))$. First, since $\pi[\alpha]\in\T^s_{t,T}$ for all $\alpha\in\A_s$ (by Definition~\ref{defn:strategy} (ii)), we observe from Lemma~\ref{lem:E = J}, Proposition~\ref{prop:xi indep. of F_t} (ii), and Proposition~\ref{prop:alpha^t,w=alpha} that 
\begin{equation}\label{K>L}
\begin{split}
K^{s,x}(s,\pi)&\ge \esssup_{\alpha\in\A_s}\E[F(\bX^{s,x,1,0,\alpha}_{\pi[\alpha]})\mid\F_s](\cdot)=\esssup_{\alpha\in\A_s}J(s,(x,1,0);\alpha^{s,\cdot},\pi[\alpha]^{s,\cdot})\\
&=\sup_{\alpha\in\A_s}J(s,(x,1,0);\alpha,\pi[\alpha])=L^\pi(s,(x,1,0)).
\end{split}
\end{equation}
On the other hand, in view of Remark~\ref{rem:esssup=lim}, there exists a sequence $\{\alpha^n\}_{n\in\N}$ in $\A$ such that $K^{s,x}(s,\pi)=\uparrow\lim_{n\to\infty}\E[F(\bX^{s,x,1,0,\alpha^n}_{\pi[\alpha^n]})\mid\F_s]$ $\overline{\P}$-a.s. By the monotone convergence theorem, 
\begin{equation}\label{E[K]<L}
\begin{split}
\E[K^{s,x}(s,\pi)]&=\E\left[\lim_{n\to\infty}\E[F(\bX^{s,x,1,0,\alpha^n}_{\pi[\alpha^n]})\mid\F_s]\right]=\lim_{n\to\infty}\E[F(\bX^{s,x,1,0,\alpha^n}_{\pi[\alpha^n]})]\\
&\le \sup_{\alpha\in\A}\E[F(\bX^{s,x,1,0,\alpha}_{\pi[\alpha]})]=L^\pi(s,(x,1,0)),
\end{split}
\end{equation} 
where the last equality is due to Lemma~\ref{lem:A=A_t}. From \eqref{K>L} and \eqref{E[K]<L}, we get $K^{s,x}(s,\pi)=L^\pi(s,(x,1,0))$. Then, for any $\alpha\in\A$, thanks to the supermartingale property introduced in Lemma~\ref{lem:supermartingale}, we have for all $r\in[s,t]$ that 
\[
L^\pi(s,(x,1,0))=K^{s,x}(s,\pi)=\Gamma^{s,x,\alpha}(s,\pi)\ge \E[\Gamma^{s,x,\alpha}(r,\pi)]\ge \E[\Gamma^{s,x,\alpha}(\pi,\pi)]\ge \E[F(\bX^{s,x,1,0,\alpha}_{\pi[\alpha]})],
\]
where the last equality follows from the fact that $K^{s,x,\alpha}(\pi,\pi)=\esssup_{\beta\in\B^{s,x,\alpha}_{\pi}}g(X^{s,x,\beta}_{\pi[\beta]})\ge g(X^{s,x,\alpha}_{\pi[\alpha]})$ $\overline{\P}$-a.s.; see \eqref{defn K}. By taking supremum over $\alpha\in\A$ and using Lemma~\ref{lem:A=A_t}, we obtain the following dynamic programming principle for $L^\tau(s,(x,1,0))$: for all $r\in[s,t]$,
\begin{equation*}
\begin{split}
L^\pi(s,(x,1,0))&=\sup_{\alpha\in\A}\E[\Gamma^{s,x,\alpha}(r,\pi)]\\
&=\sup_{\alpha\in\A}\E\left[\int_s^r Y^{s,x,1,\alpha}_uf(u,X^{s,x,\alpha}_u,\alpha_u)du +Y^{s,x,1,\alpha}_r L^\pi(r,(X^{s,x,\alpha}_r,1,0))\right],
\end{split}
\end{equation*}
where the second equality follows from the fact $K^{s,x,\alpha}(r,\pi)=K^{r,X^{s,x,\alpha}_r,\alpha}(r,\pi)=L^\pi(r,(X^{s,x,\alpha}_r,1,0))$ $\overline{\P}$-a.s., as a consequence of Lemma~\ref{lem:K^s=K^r}. Now, we may apply the same argument in \cite[Lemma 3.3.7]{Krylov-book-80} to show that $L^\pi(s,(x,1,0))$ is continuous in $s$ on $[0,t]$.
\end{proof}

\begin{prop}\label{prop:WDPP_sub}
Fix $(t,\bx)\in[0,T]\times\mathcal{S}$ and $\eps > 0$. For any $\pi\in\Pi_{t,T}$ and $\varphi\in \operatorname{LSC}([0,T]\times\mathbb{R}^d)$ with $\varphi\ge U$,
there exists $\pi^*\in\Pi_{t,T}$ such that 
\begin{equation*}
\E\left[F(\bX^{t,\bx,\alpha}_{\pi^*[\alpha]})\right] \le \E\left[Y^{t,x,y,\alpha}_{\pi[\alpha]}\varphi\left(\pi[\alpha],X^{t,x,\alpha}_{\pi[\alpha]}\right)+Z^{t,x,y,z,\alpha}_{\pi[\alpha]}\right]+3\eps,\ \ \ \forall \alpha\in\A.
\end{equation*}
\end{prop}

\begin{proof}
For each $(s,\eta)\in [0,T]\times\mathcal{S}$, by the definition of $\bar{U}$, there exists $\pi^{(s,\eta),\eps}\in\Pi_{s,T}$ such that 
\begin{equation}\label{sup<U}
\sup_{\alpha\in\A_s}J\left(s,\eta;\alpha,\pi^{(s,\eta),\eps}[\alpha]\right) \le \bar{U}(s,\eta)+\eps.
\end{equation}  
Recall the definition of $\bar{\varphi}$ in \eqref{def barphi} and note that $\varphi\in \operatorname{LSC}([0,T]\times\mathbb{R}^d)$ implies $\bar{\varphi}\in \operatorname{LSC}([0,T]\times\mathcal{S})$. Then, by the lower semicontinuity of $\bar{\varphi}$ on $[0,T]\times\mathcal{S}$ and the upper semicontinuity of $L^{\pi^{(s,\eta),\eps}}$ on $[0,s]\times\mathcal{S}$ (from Lemma~\ref{lem:L continuous}), there must exist $r^{(s,\eta)}>0$ such that
\[
\bar{\varphi}(t',x')-\bar{\varphi}(s,\eta) \ge -\eps\ \text{and}\ L^{\pi^{(s,\eta),\eps}}(t',x')-L^{\pi^{(s,\eta),\eps}}(s,\eta)\le\eps, 
\]
for any $(t',x')$ contained in the ball $B(s,\eta;r^{(s,\eta)})$, defined as in \eqref{defn ball}. It follows that if $(t',x')\in B(s,\eta;r^{(s,\eta)})$, we have
\begin{equation*}
L^{\pi^{(s,\eta),\eps}}(t',x') \le L^{\pi^{(s,\eta),\eps}}(s,\eta)+\eps \le \bar{U}(s,\eta)+2\eps \le \bar{\varphi}(s,\eta)+2\eps \le \bar{\varphi}(t',x')+3\eps,
\end{equation*}
where the second inequality is due to \eqref{sup<U}. By the same construction in the proof of Proposition~\ref{prop:WDPP_super}, there exists a countable covering $\{B(t_i,x_i;r_i)\}_{i\in\N}$ of $(0,T]\times\mathcal{S}$, from which we can take a countable disjoint covering $\{A_i\}_{i\in\N\cup\{0\}}$ of $(0,T]\times\mathcal{S}$ such that 
\begin{equation}\label{Ai structure 2}
\begin{split}
&(\pi[\alpha],\bX^{t,\bx,\alpha}_{\pi[\alpha]})\in\cup_{i=1}^\ell A_i\ \overline{\P}\hbox{-a.s.}\ \forall \alpha\in\A,\\ 
&L^{\pi^{i,\eps}}(t',x') \le \bar{\varphi}(t',x')+3\eps\ \hbox{for}\ (t',x')\in A_i,\ \hbox{where}\ \pi^{i,\eps}:=\pi^{(t_i,x_i),\eps}.
\end{split}
\end{equation}
Now, define $\pi^*\in\Pi_{t,T}$ by
\[
\pi^{*}[\alpha] := \sum_{i\ge 1}\pi^{i,\eps}[\alpha]1_{A_i}(\pi[\alpha],\bX^{t,\bx,\alpha}_{\pi[\alpha]}),\ \ \forall \alpha\in\A.
\]
Fix $\alpha\in\A_t$. Observe that for $\overline{\P}$-a.e. $\omega\in \left\{(\pi[\alpha],\bX^{t,\bx,\alpha}_{\pi[\alpha]})\in A_i\right\}\subseteq\{\pi[\alpha]\le t_i\}$, Definition~\ref{defn:strategy} (iii) gives 
\begin{align}\label{shifted pi}
(\pi^{i,\eps}[\alpha])^{\pi[\alpha],\omega}(\omega')=\pi^{i,\eps}[\alpha^{\pi[\alpha],\omega}](\omega')\ \ \hbox{for}\ \overline{\P}\hbox{-a.e.}\ \omega'\in\Omega.
\end{align}
We then deduce from Lemma~\ref{lem:E = J}, \eqref{shifted pi}, \eqref{L}, 
and \eqref{Ai structure 2} that for $\overline{\P}$-a.e. $\omega\in\Omega$,
\begin{equation*}
\begin{split}
&\E\left[F(\bX^{t,\bx,\alpha}_{\pi^{*}[\alpha]})\ \middle|\ \F_{\pi[\alpha]}\right](\omega)\ 1_{A_i}(\pi[\alpha](\omega),\bX^{t,\bx,\alpha}_{\pi[\alpha]}(\omega))\\
&\hspace{0.1in} = J\left(\pi[\alpha](\omega),\bX^{t,\bx,\alpha}_{\pi[\alpha]}(\omega);\alpha^{\pi[\alpha],\omega},\pi^{i,\eps}[\alpha^{\pi[\alpha],\omega}]\right)1_{A_i}(\pi[\alpha](\omega),\bX^{t,\bx,\alpha}_{\pi[\alpha]}(\omega))\\
&\hspace{0.1in}\le
L^{\pi^{i,\eps}}\left(\pi[\alpha](\omega),\bX^{t,\bx,\alpha}_{\pi[\alpha]}(\omega)\right)1_{A_i}(\pi[\alpha](\omega),\bX^{t,\bx,\alpha}_{\pi[\alpha]}(\omega))\\
&\hspace{0.1in}\le\left[\bar{\varphi}\left(\pi[\alpha](\omega),\bX^{t,\bx,\alpha}_{\pi[\alpha]}(\omega)\right)+3\eps\right]1_{A_i}(\pi[\alpha](\omega),\bX^{t,\bx,\alpha}_{\pi[\alpha]}(\omega)).
\end{split}
\end{equation*}
It follows from the monotone convergence theorem that
\begin{equation*}\label{<phi 2}
\E\left[F(\bX^{t,\bx,\alpha}_{\pi^{*}[\alpha]})\right] =
\sum_{i\ge 1}\E\left[\E\left[F(\bX^{t,\bx,\alpha}_{\pi^{*}[\alpha]})\ \middle|\ \F_{\pi[\alpha]}\right] 1_{A_i}(\pi[\alpha],\bX^{t,\bx,\alpha}_{\pi[\alpha]})\right]\nonumber\\
\le \E\left[\bar{\varphi}(\pi[\alpha],\bX^{t,\bx,\alpha}_{\pi[\alpha]})\right]+3\eps,
\end{equation*}
which is the desired result by recalling again the definition of $\bar{\varphi}$ in \eqref{def barphi}.
\end{proof}



The following is the main result of this section. Recall that the operator $H$ is defined in \eqref{H}, and $H_*$ denotes the lower semicontinuous envelope of $H$ defined as in \eqref{envelopes}.

\begin{prop}\label{prop:vis sub}
The function $U^*$ is a viscosity subsolution to the obstacle problem of a Hamilton-Jacobi-Bellman equation
\begin{equation*}
\max\left\{ c(t,x)w-\frac{\partial w}{\partial t}+H_*(t,x,D_x w,D^2_{x} w), w-g(x)\right\}=0\  \hbox{on}\ [0,T)\times\R^d,
\end{equation*}
and satisfies the polynomial growth condition \eqref{poly grow U,V}. 
\end{prop}

\begin{proof}
We may argue as in the proof of Proposition~\ref{prop:vis super} to show that $U^*$ satisfies \eqref{poly grow U,V}. To prove the subsolution property, we
assume the contrary that there exist $h\in C^{1,2}([0,T)\times\mathbb{R}^d)$ and $(t_0,x_0)\in[0,T)\times\mathbb{R}^d$ satisfying 
\[0=(U^*-h)(t_0,x_0)>(U^*-h)(t,x),\ \text{for any}\ (t,x)\in[0,T)\times\mathbb{R}^d,\ (t,x)\neq(t_0,x_0),\]
such that 
\begin{equation}\label{HJB>0}
\max\left\{ c(t_0,x_0)h(t_0,x_0)-\frac{\partial h}{\partial t}(t_0,x_0)+H_*(\cdot,D_x h ,D^2_{x} h)(t_0,x_0), h(t_0,x_0)-g(x_0)\right\}>0.
\end{equation}
Since $U^*(t_0,x_0)=h(t_0,x_0)$ and $U\leq g$ by definition, continuity of $g$ implies that $h(t_0,x_0)=U^*(t_0,x_0)\leq g(x_0)$. Therefore, we can conclude from \eqref{HJB>0} that
\begin{equation}\label{>0}
c(t_0,x_0)h(t_0,x_0)-\frac{\partial h}{\partial t}(t_0,x_0)+H_*(\cdot,D_x h ,D^2_{x} h)(t_0,x_0)>0.
\end{equation}
Define the function $\tilde{h}$ by
\[
\tilde{h}(t,x):=h(t,x)+\eps(|t-t_0|^2+|x-x_0|^4).
\]
Note that $(\tilde{h},\partial_t\tilde{h},D_x\tilde{h},D_x^2\tilde{h})(t_0,x_0)=(h,\partial_t h,D_x h,D_x^2 h)(t_0,x_0)$. Then, by the lower semicontinuity of $H_*$, there exists $r>0$ with $t_0+r<T$ such that
\begin{equation} \label{>0 in B_r}
c(t,x)\tilde{h}(t,x)-\frac{\partial \tilde{h}}{\partial t}(t,x)+H^a(\cdot,D_x \tilde{h} ,D^2_{x} \tilde{h})(t,x)>0,\ \forall\ a\in M\ \text{and}\ (t,x)\in \bar{B}_r(t_0,x_0).
\end{equation}
Now, define $\eta>0$ by 
\begin{equation}\label{eta}
\eta e^{\bar{c}T} :=\min_{\partial B_r(t_0,x_0)}(\tilde{h}-h)>0.
\end{equation}
Take $(\hat{t},\hat{x})\in B_r(t_0,x_0)$ such that $|(U-\tilde{h})(\hat{t},\hat{x})|<\eta/2$, and define $\pi\in\Pi_{\hat{t},T}$ by
\begin{equation}\label{def pi^circ}
\pi[\alpha]:=\inf\left\{s\geq\hat{t}\ \middle|\ (s,X^{\hat{t},\hat{x},\alpha}_s)\notin B_r(t_0,x_0)\right\},\ \forall \alpha\in\A.
\end{equation}
For any $\alpha\in\A_{\hat{t}}$, applying the product rule of stochastic calculus to $Y^{\hat{t},\hat{x},1,\alpha}_s \tilde{h}(s,X^{\hat{t},\hat{x},\alpha}_s)$, we get
\begin{equation*}
\begin{split}
\tilde{h}(\hat{t},\hat{x}) &= \E\bigg[Y^{\hat{t},\hat{x},1,\alpha}_{\pi[\alpha]} \tilde{h}(\pi[\alpha],X^{\hat{t},\hat{x},\alpha}_{\pi[\alpha]})\\
&\hspace{0.5in}+ \int_{\hat{t}}^{\pi[\alpha]} Y^{\hat{t},\hat{x},1,\alpha}_s\bigg(c\tilde{h}-\frac{\partial \tilde{h}}{\partial t}+H^{\alpha_s}(\cdot,D_x \tilde{h}, D^2_{x} \tilde{h})+f\bigg)(s,X^{\hat{t},\hat{x},\alpha}_{s},\alpha_s) ds \bigg]\\
&>\E\bigg[Y^{\hat{t},\hat{x},1,\alpha}_{\pi[\alpha]} h(\pi[\alpha],X^{\hat{t},\hat{x},\alpha}_{\pi[\alpha]}) + \int_{\hat{t}}^{\pi[\alpha]} Y^{\hat{t},\hat{x},1,\alpha}_s f(s,X^{\hat{t},\hat{x},\alpha}_{s},\alpha_s) ds \bigg]+\eta,
\end{split}
\end{equation*}
where the inequality follows from \eqref{eta}, \eqref{>0 in B_r} and $c\leq \bar{c}$. Moreover, by our choice of $(\hat{t},\hat{x})$, we have $U(\hat{t},\hat{x})+\eta/2>\tilde{h}(\hat{t},\hat{x})$. It follows that
\begin{equation}\label{V > E+4eps}
U(\hat{t},\hat{x}) > 
\E\left[Y^{\hat{t},\hat{x},1,\alpha}_{\pi[\alpha]} h(\pi[\alpha],X^{\hat{t},\hat{x},\alpha}_{\pi[\alpha]}) + \int_{\hat{t}}^{\pi[\alpha]} Y^{\hat{t},\hat{x},1,\alpha}_s f(s,X^{\hat{t},\hat{x},\alpha}_{s},\alpha_s) ds \right]+\frac{\eta}{2},\ \hbox{for any}\ \alpha\in\A_{\hat{t}}.
\end{equation}
Finally, we conclude from the definition of $U$ and Proposition~\ref{prop:WDPP_sub} that there exist $\pi^*\in\Pi_{\hat{t},T}$ and $\hat{\alpha}\in\A_{\hat{t}}$ such that  
\begin{equation*}
\begin{split}
U(\hat{t},\hat{x})=\bar{U}(\hat{t},\hat{x},1,0) &\le \sup_{\alpha\in\A_{\hat{t}}}\E\left[F\left(\bX^{\hat{t},\hat{x},1,0,\alpha}_{\pi^*[\alpha]}\right)\right]\le \E\left[F\left(\bX^{\hat{t},\hat{x},1,0,\hat{\alpha}}_{\pi^*[\hat{\alpha}]}\right)\right]+\frac{\eta}{4}\\
&\le  \E\left[Y^{\hat{t},\hat{x},1,\hat{\alpha}}_{\pi[\hat{\alpha}]}h(\pi[\hat{\alpha}],X^{\hat{t},\hat{x},\hat{\alpha}}_{\pi[\hat{\alpha}]})+Z^{\hat{t},\hat{x},1,0,\hat{\alpha}}_{\pi[\hat{\alpha}]}\right]+\frac{\eta}{2},
\end{split}
\end{equation*}
which contradicts \eqref{V > E+4eps}.
\end{proof}

\section{Comparison}\label{sec:comparison}
In this section, to state an appropriate comparison result, we assume a stronger version of \eqref{glo Lips 1} as follows: there exists $K>0$ such that for any $t,s\in[0,T],\ x,y\in\mathbb{R}^d$, and $u\in M$, 
\begin{equation}\label{glo Lips 2}
|b(t,x,u)-b(s,y,u)|+|\sigma(t,x,u)-\sigma(s,y,u)| \leq K(|t-s|+|x-y|). 
\end{equation}
Moreover, we impose an additional condition on $f$:
\begin{equation}\label{uniform conti. f}
f(t,x,u)\ \hbox{is uniformly continuous in}\ (t,x),\ \hbox{uniformly in}\ u\in M.
\end{equation}
Note that the conditions \eqref{glo Lips 2} and \eqref{uniform conti. f}, together with the linear growth condition \eqref{linear grow} on $b$ and $\sigma$, imply that the operator $H$ defined in \eqref{H} is continuous, and $\overline{H}=H=H_*$. 

\begin{prop}\label{prop:comparison}
Assume \eqref{glo Lips 2} and \eqref{uniform conti. f}. Let $u$ (resp. $v$) be an upper semicontinuous viscosity subsolution (resp. a lower semicontinuous viscosity supersolution), with polynomial growth in $x$, to 
\begin{equation}
\max\left\{ c(t,x)w-\frac{\partial w}{\partial t}+H(t,x,D_x w, D^2_x w),\ w-g(x)\right\}=0\ \hbox{on}\ [0,T)\times\R^d,
\end{equation}
and $u(T,x)\le v(T,x)$ for all $x\in\R^d$. Then $u\le v$ on $[0,T)\times\R^d$.
\end{prop}

\begin{proof}
For $\lambda>0$, define $u^\lambda:=e^{\lambda t}u(t,x)$, $v^\lambda:=e^{\lambda t}v(t,x)$, and 
\[
H_\lambda(t,x,p,A):=\inf_{a\in M}\left\{-b(t,x,a)\cdot p-\frac{1}{2}Tr[\sigma\sigma'(t,x,a)A]-e^{\lambda t}f(t,x,a)\right\}.
\]
Note that the conditions \eqref{glo Lips 2} and \eqref{uniform conti. f}, together with the linear growth condition \eqref{linear grow} on $b$ and $\sigma$ and the polynomial growth condition \eqref{poly grow} on $f$, imply that $H_\lambda$ is continuous. By definition, $u$ (resp. $v$) is upper semicontinuous (resp. lower semicontinuous) and has polynomial growth. Moreover, by direct calculations, the subsolution property of $u$ (resp. supersolution property of $v$) implies that $u^\lambda$ (resp. $v^\lambda$) is a viscosity subsolution (resp. viscosiy supersolution) to 
\begin{equation}\label{HJB obstacle 3}
\max\left\{ \left(c(t,x)+\lambda\right)w-\frac{\partial w}{\partial t}+H_\lambda(t,x,D_x w,D^2_{x} w),\ w-e^{\lambda t}g(x)\right\}=0\  \hbox{on}\ [0,T)\times\R^d.
\end{equation}
For any $(t,x,r,q,p,A)\in[0,T]\times\R^d\times\R\times\R\times\R^d\times\mathbb{M}^d$, define
\[
F_1(t,x,r,q,p,A):=\left(c(t,x)+\lambda\right)r-q+H_\lambda(t,x,p,A)\ \hbox{and}\ F_2(t,x,r):=r-e^{\lambda t}g(x).
\]
Since $F_1$ and $F_2$ are by definition continuous, so is $F_3:=\max\{F_1,F_2\}$. We can then write \eqref{HJB obstacle 3} as $F_3(t,x,w,\frac{\partial w}{\partial t}, D_x w,D_x^2 w)=0$. 

From the polynomial growth condition on $u^\lambda$ and $v^\lambda$, there exists some $p>0$ such that 
\[
\sup_{[0,T]\times\R^d}\frac{|u^\lambda(t,x)|+|v^\lambda(t,x)|}{1+|x|^p}<\infty.
\]
Define $\gamma(x):=1+|x|^{2p}$ and set $\varphi(t,x):=e^{-\lambda t}\gamma(x)$. From the linear growth condition \eqref{linear grow} on $b$ and $\sigma$, a direct calculation shows that $|b(t,x,a)\cdot D_x\gamma+\frac{1}{2}Tr[\sigma\sigma'(t,x,a)D_x^2\gamma]|\le C\gamma(x)$ for some $C>0$. It follows that
\begin{equation}\label{phi supersol.}
\begin{split}
&(c(t,x)+\lambda)\varphi-\frac{\partial\varphi}{\partial t}+\inf_{a\in M}\left\{-b(t,x,a)D_x\varphi-\frac{1}{2}Tr[\sigma\sigma'(t,x,a)D_x^2\varphi]\right\}\\
&\hspace{0.3in}=e^{-\lambda t}\left([c(t,x)+2\lambda]\gamma+\inf_{a\in M}\left\{-b(t,x,a)D_x\gamma-\frac{1}{2}Tr[\sigma\sigma'(t,x,a)D_x^2\gamma]\right\}\right)\\
&\hspace{0.3in}\ge e^{-\lambda t}[c(t,x)+2\lambda-C]\gamma\ge 0,\ \hbox{if}\ \lambda\ge\frac{C}{2}.
\end{split}
\end{equation} 
Now, take $\lambda\ge\frac{C}{2}$ and define $v^\lambda_\eps:=v^\lambda+\eps\varphi$ for all $\eps>0$. By definition, $v^\lambda_\eps$ is lower semicontinuous. Given any $h\in C^{1,2}([0,T)\times\R^d)$ and $(t_0,x_0)\in[0,T)\times\R^d$ such that $v^\lambda_\eps-h$ attains a local minimum, which equals $0$, at $(t_0,x_0)$, the supersolution property of $v^\lambda$ implies either $F_1\left(\cdot,h(\cdot),\frac{\partial h}{\partial t}(\cdot),D_x h(\cdot),D_x^2 h(\cdot)\right)(t_0,x_0)\ge 0$ or $F_2\left(\cdot,h(\cdot)\right)(t_0,x_0)\ge 0$. If the former holds true, we see from \eqref{phi supersol.} that 
\[
F_1\left(\cdot,v^\lambda_\eps(\cdot),\frac{\partial v^\lambda_\eps}{\partial t}(\cdot),D_x v^\lambda_\eps(\cdot),D_x^2 v^\lambda_\eps(\cdot)\right)(t_0,x_0)\ge 0;
\]
if the latter holds true, then $F_2\left(\cdot,v^\lambda_\eps(\cdot)\right)(t_0,x_0)=v^\lambda_\eps(t_0,x_0)-e^{\lambda t_0}g(x_0)=F_2\left(\cdot,v^\lambda(\cdot)\right)(t_0,x_0)+\eps\varphi(t_0,x_0)=F_2\left(\cdot,h(\cdot)\right)(t_0,x_0)+\eps\varphi(t_0,x_0)\ge 0$. Therefore, $v^\lambda_\eps$ is a lower semicontinuous viscosity supersolution to \eqref{HJB obstacle 3}. 

We would like to show $u^\lambda\le v^\lambda_\eps$ on $[0,T)\times\R^d$ for all $\eps>0$; then by sending $\eps$ to $0$, we can conclude $u\le v$ on $[0,T)\times\R^d$, as desired. We will argue by contradiction, and thus assume that 
\[
N:=\sup_{[0,T]\times\R^d}(u^\lambda-v^\lambda_\eps)(t,x)>0
\]
From the polynomial growth condition on $u^\lambda$ and $v^\lambda$ and the definition of $\varphi$, we have 
\[
\lim_{|x|\to\infty}\sup_{[0,T]}(u^\lambda-v^\lambda_\eps)(t,x)=-\infty.
\]
It follows that there exists some bounded open set $\mathcal{O}\subset\R^d$ such that the maximum $N$ is attained at some point contained in $[0,T]\times\mathcal{O}$. For each $\delta>0$, define the functions
\[
\Phi_\delta(t,s,x,y):=u^\lambda(t,x)-v^\lambda_\eps(s,y)-\eta_\delta(t,s,x,y),\ \hbox{with}\ \eta_\delta(t,s,x,y):=\frac{1}{2\delta}[|t-s|^2+|x-y|^2].
\]
Since $\Phi_\delta$ is upper semicontinuous, it attains its maximum, denoted by $N_\delta$, on the compact set $[0,T]^2\times\overline{\mathcal{O}}^2$ at some point $(t_\delta,s_\delta,x_\delta,y_\delta)$. Then, the upper semicontinuity of $u^\lambda(t,x)-v^\lambda_\eps(s,y)$ implies that $\left(u^\lambda(t_\delta,x_\delta)-v^\lambda_\eps(s_\delta,y_\delta)\right)_{\delta}$ is bounded above; moreover, it is also bounded below as
\begin{equation}\label{N<u-v}
N\le N_\delta= u^\lambda(t_\delta,x_\delta)-v^\lambda_\eps(s_\delta,y_\delta)-\eta_\delta(t_\delta,s_\delta,x_\delta,y_\delta)\le u^\lambda(t_\delta,x_\delta)-v^\lambda_\eps(s_\delta,y_\delta).
\end{equation}
Then we see from \eqref{N<u-v} and the boundedness of $\left(u^\lambda(t_\delta,x_\delta)-v^\lambda_\eps(s_\delta,y_\delta)\right)_{\delta}$ that $\left(\eta_\delta(t_\delta,s_\delta,x_\delta,y_\delta)\right)_\delta$ is also bounded. Now, note that the bounded sequence $\left(t_\delta,s_\delta,x_\delta,y_\delta\right)_\delta$ converges, up to a subsequence, to some point $(\tilde{t},\tilde{s},\tilde{x},\tilde{y})\in[0,T]^2\times\overline{\mathcal{O}}^2$. Then the definition of $\eta_\delta$ and the boundedness of $(\eta_\delta(t_\delta,s_\delta,x_\delta,y_\delta))_\delta$ imply that $\tilde{t}=\tilde{s}$ and $\tilde{x}=\tilde{y}$. Then, by sending $\delta$ to $0$ in \eqref{N<u-v}, we see that the last expression becomes $(u^\lambda-v^\lambda_\eps)(\tilde{t},\tilde{x})\le N$, which implies that 
\begin{equation}\label{N_delta to N}
N_\delta\to N\ \hbox{and}\ \eta_\delta(t_\delta,s_\delta,x_\delta,y_\delta)\to 0.  
\end{equation}

In view of Ishii's Lemma (see e.g. \cite[Lemma 4.4.6]{Pham-book}) and \cite[Remark 4.4.9]{Pham-book}, for each $\delta>0$, there exist $A_\delta, B_\delta\in\mathbb{M}^d$ such that
\begin{equation}\label{trace estimate}
Tr(CC'A_\delta-DD'B_\delta)\le\frac{3}{\delta}|C-D|^2\ \hbox{for all}\ C,D\in\mathbb{M}^d,
\end{equation}
and
\[
\left(\frac{1}{\delta}(t_\delta-s_\delta),\frac{1}{\delta}(x_\delta-y_\delta),A_\delta\right)\in\bar{\mathcal{P}}^{2,+}u^\lambda(t_\delta,x_\delta),\ \left(\frac{1}{\delta}(t_\delta-s_\delta),\frac{1}{\delta}(x_\delta-y_\delta),B_\delta\right)\in\bar{\mathcal{P}}^{2,-}v^\lambda_\eps(s_\delta,y_\delta),
\]
where $\bar{\mathcal{P}}^{2,+}w(t,x)$ (resp. $\bar{\mathcal{P}}^{2,-}w(t,x)$) denotes the superjet (resp. subjet) of an upper semicontinuous (resp. a lower semicontinuous) function $w$ at $(t,x)\in[0,T]\times\R^d$; for the definition of these notions, see e.g. \cite{User's-guide-viscosity} and \cite{Pham-book}. Since the function $F_3=\max\{F_1,F_2\}$ is continuous, we may apply \cite[Lemma 4.4.5]{Pham-book} and obtain that 
\begin{equation*}
\begin{split}
\max\left\{ \left(c(t_\delta,x_\delta)+\lambda\right)u^\lambda(t_\delta,x_\delta)-\frac{1}{\delta}(t_\delta-s_\delta)+H_{\lambda}(t_\delta,x_\delta,\frac{1}{\delta}(x_\delta-y_\delta), A_\delta),\ u^\lambda(t_\delta,x_\delta)-e^{\lambda t_\delta}g(x_\delta)\right\}&\le 0,\\
\max\left\{\left(c(s_\delta,y_\delta)+\lambda\right)v^\lambda_\eps(s_\delta,y_\delta)-\frac{1}{\delta}(t_\delta-s_\delta)+H_{\lambda}(s_\delta,y_\delta,\frac{1}{\delta}(x_\delta-y_\delta), B_\delta),\ v^\lambda_\eps(s_\delta,y_\delta)-e^{\lambda s_\delta}g(y_\delta)\right\}&\ge 0.
\end{split}
\end{equation*}
Noting that $\max\{a,b\}-\max\{c,d\}\ge\min\{a-c,b-d\}$ for any $a,b,c,d\in\R$, we then have
\begin{equation}\label{min}
\begin{split}
\min\bigg\{&\left(c(t_\delta,x_\delta)+\lambda\right)u^\lambda(t_\delta,x_\delta)-\left(c(s_\delta,y_\delta)+\lambda\right)v^\lambda_\eps(s_\delta,y_\delta)+H_{\lambda}(t_\delta,x_\delta,\frac{1}{\delta}(x_\delta-y_\delta), A_\delta)\\
&-H_{\lambda}(s_\delta,y_\delta,\frac{1}{\delta}(x_\delta-y_\delta), B_\delta),\ u^\lambda(t_\delta,x_\delta)-v^\lambda_\eps(s_\delta,y_\delta)+e^{\lambda s_\delta}g(y_\delta)-e^{\lambda t_\delta}g(x_\delta)\bigg\}\le 0.
\end{split}
\end{equation}
Since $u^\lambda(t_\delta,x_\delta)-v^\lambda_\eps(s_\delta,y_\delta)+e^{\lambda s_\delta}g(y_\delta)-e^{\lambda t_\delta}g(x_\delta)=N_\delta+\eta_\delta(t_\delta,s_\delta,x_\delta,y_\delta)+e^{\lambda s_\delta}g(y_\delta)-e^{\lambda t_\delta}g(x_\delta)\to N>0$, we conclude from \eqref{min} that as $\delta$ small enough, we must have 
\begin{equation*}
\begin{split}
&\left(c(t_\delta,x_\delta)+\lambda\right)u^\lambda(t_\delta,x_\delta)-\left(c(s_\delta,y_\delta)+\lambda\right)v^\lambda_\eps(s_\delta,y_\delta)\\
&\le H_{\lambda}(s_\delta,y_\delta,\frac{1}{\delta}(x_\delta-y_\delta), B_\delta)-H_{\lambda}(t_\delta,x_\delta,\frac{1}{\delta}(x_\delta-y_\delta), A_\delta)\le\mu(|t_\delta-s_\delta|+|x_\delta-y_\delta|+\frac{3}{\delta}|x_\delta-y_\delta|^2),
\end{split}
\end{equation*}
for some function $\mu$ such that $\mu(z)\to 0$ as $z\to 0$; note that the second inequality follows from \eqref{glo Lips 2}, \eqref{uniform conti. f}, and \eqref{trace estimate}. Finally, by sending $\delta$ to $0$ and using \eqref{N_delta to N}, we get $(c(\tilde{t},\tilde{x})+\lambda)N\le 0$, a contradiction.
\end{proof}

Now, we turn to the behavior of $V_*$, the lower semicontinuous envelope of $V$ defined as in \eqref{envelopes}, at terminal time $T$. 

\begin{lem}\label{lem:V_*>g}
For all $x\in\R^d$, $V_*(T,x)\ge g(x)$. 
\end{lem}

\begin{proof}
Fix $\alpha\in\A$. Take an arbitrary sequence $(t_m,x_m)\to(T,x)$ with $t_m<T$ for all $m\in\N$. By the definition of $V$, we can choose for each $m\in\N$ a stopping time $\tau_m\in\T^{t_m}_{t_m,T}$ such that 
\begin{equation*}
\begin{split}
V(t_m,x_m)&\ge \inf_{\tau\in\T^{t_m}_{t_m,T}}\E\left[\int^{\tau}_{t_m}Y^{t_m,x_m,1,\alpha}f(s,X^{t_m,x_m,\alpha},\alpha_s)ds+Y^{t_m,x_m,1,\alpha}_\tau g(X^{t_m,x_m,\alpha}_\tau)\right]\\
&\ge \E\left[\int^{\tau_m}_{t_m}Y^{t_m,x_m,1,\alpha}f(s,X^{t_m,x_m,\alpha},\alpha_s)ds+Y^{t_m,x_m,1,\alpha}_{\tau_m}g(X^{t_m,x_m,\alpha}_{\tau_m})\right]-\frac{1}{m}.
\end{split}
\end{equation*}
Note that $\tau_m\to T$ as $\tau_m\in\T^{t_m}_{t_m,T}$ and $t_m\to T$. Then it follows from Fatou's lemma that $\liminf_{m\to\infty}V(t_m,x_m)\ge g(x)$. Since $(t_m,x_m)$ is arbitrarily chosen, we conclude $V_*(T,x)\ge g(x)$.
\end{proof}


\begin{thm}\label{prop:U=V}
Assume \eqref{glo Lips 2} and \eqref{uniform conti. f}. Then $U^*=V$ on $[0,T]\times\R^d$. In particular, $U=V$ on $[0,T]\times\R^d$, i.e. the game has a value, which is the unique viscosity solution to \eqref{HJB obstacle 1} with terminal condition $w(T,x)=g(x)$ for $x\in\R^d$.
\end{thm}

\begin{proof}
Since by definition $U(t,x)\le g(x)$ on $[0,T]\times\R^d$, we have $U^*(t,x)\le g(x)$ on $[0,T]\times\R^d$ by the continuity of $g$. Then by Lemma~\ref{lem:V_*>g} and the fact that $U^*\ge U\ge V$, we have $U^*(T,x)=V(T,x)=g(x)$ for all $x\in\R^d$. Recall that under \eqref{glo Lips 2} and \eqref{uniform conti. f}, the function $H$ is continuous, and $\overline{H}=H=H_*$. Now, in view of Propositions~\ref{prop:vis super} and ~\ref{prop:vis sub}, and the fact that $U^*(T,\cdot)=V(T,\cdot)$ and $\overline{H}=H=H_*$, we conclude from Proposition~\ref{prop:comparison} that $U^*=V$ on $[0,T]\times\R^d$, which in particular implies $U=V$ on $[0,T]\times\R^d$.  
\end{proof}


\appendix
\section{Proofs for Sections~\ref{sec:prelim} and~\ref{sec:problem}}\label{sec:appendix}
This Appendix is devoted to rigorous proofs of Propositions~\ref{prop:xi indep. of F_t},~\ref{lem: tau a stopping time a.s.},~\ref{prop:alpha^t,w=alpha}, and Lemma~\ref{lem:E = J}. To this end, we will first derive several auxiliary results.

Recall the definitions introduced in Subsection~\ref{subsec:setup}. Fix $t\in[0,T]$. For any $A\subseteq\Omega$, $\tilde{A}\subseteq\Omega^t$, and $x\in\R^d$, we set
\[
\tilde{A}_x:=\{\tilde{\omega}\in \tilde{A}\mid\tilde{\omega}_t=x\},
\]
and define 
\[
A^{t,\omega}:=\{\tilde{\omega}\in\Omega^t \mid \omega\otimes_t\tilde{\omega}\in A\},\ \ \ \ A^{t,\omega}_x:=(A^{t,\omega})_x,\ \ \ \   \omega\otimes_t\tilde{A}:=\{\omega\otimes_t\tilde{\omega}\mid\tilde{\omega}\in\tilde{A}\}.
\] 
Given a random time $\tau:\Omega\mapsto[0,\infty]$, whenever $\omega\in\Omega$ is fixed, we simplify our notation as $A^{\tau,\omega}=A^{\tau(\omega),\omega}$.
We also consider
\begin{equation}\label{H in g}
\mathcal{H}^t_s:=\psi^{-1}_t \G^{t,0}_s \subseteq\G^t_s,\ \forall s\in[t,T].
\end{equation}
Note that the inclusion follows from the Borel measurability of $\psi_t$. Finally, while $\E$ denotes the expectation taken under $\overline{\P}$, in this appendix we also consider $\E_\P$, the expectation taken under $\P$.   

\begin{lem}\label{lem:measurability}
Fix $t\in [0,T]$ and $\omega\in\Omega$. For any $r\in[t,T]$, $A\in\G_r$, $\tilde{A}\in\G^t_r$, and $\xi\in L^0(\Omega,\G_r)$,
\begin{itemize}
\item [(i)] $A^{t,\omega}_x = A^{t,\omega}_0+x$ and $A^{t,\omega}_x\in\G^{t,x}_r,\ \forall x\in\R^d$. 
\item [(ii)] $A^{t,\omega}=\psi_t^{-1}A^{t,\omega}_0\in
\cH^t_r\subseteq\G^t_r$ and $\P^t(A^{t,\omega})=\P^{t,x}(A^{t,\omega}_x)
=\P^{t,x}(A^{t,\omega}),\ \forall x\in\R^d$.
\item [(iii)] $\phi^{-1}_t A^{t,\omega}\in\phi^{-1}_t\cH^t_r\subseteq
\G_r$ 
and $\P(\phi_t^{-1}A^{t,\omega})=\P^t(A^{t,\omega})$.
\item [(iv)] $\omega\otimes_t\tilde{A}_{\omega_t}\in\G_r$. Hence, $\omega\otimes_t A^{t,\omega}_{\omega_t}\in\G_r$.
\item [(v)] For any Borel subset $\mathcal{E}$ of $\R$, 
$(\xi^{t,\omega})^{-1}(\mathcal{E})\in\phi^{-1}_t\cH^t_r\subseteq\G_r$. Hence, 
$\xi^{t,\omega}\in L^0(\Omega,\G_r)$.
\end{itemize}
\end{lem}

\begin{proof}
(i) Fix $x\in\R^d$. Since $\tilde{\omega}\in A^{t,\omega}_0\Leftrightarrow\omega\otimes_t\tilde{\omega}\in A$ and $\tilde{\omega}_t=0\Leftrightarrow(\omega\otimes_t(\tilde{\omega}+x))_\cdot=\omega_\cdot1_{[0,t]}(\cdot)+((\tilde{\omega}_\cdot+x)-(\tilde{\omega}_t+x)+\omega_t)1_{(t,T]}(\cdot)=(\omega\otimes_t\tilde{\omega})_\cdot\in A$ and $(\tilde{\omega}+x)_t=x\Leftrightarrow \tilde{\omega}+x\in A^{t,\omega}_x$, we conclude $A^{t,\omega}_x = A^{t,\omega}_0+x$. 

Set $\Lambda:=\{A\subseteq \Omega \mid A^{t,\omega}_x\in\G^{t,x}_r\}$. Note that $\Omega\in\Lambda$ since $\Omega^{t,\omega}_x=\{\tilde{w}\in\Omega^t \mid \omega\otimes_t\tilde{\omega}\in \Omega, \tilde{\omega}_t=x\}=(\Omega^t)_x\in\G^{t,x}_r$. Given $A\in\Lambda$, we have $(A^c)^{t,\omega}_x=(\Omega^t)_x\setminus\{\tilde{\omega}\in\Omega^t \mid \omega\otimes_t \tilde{\omega}\in A,\tilde{\omega}_t=x\}=(\Omega^t)_x\setminus A^{t,\omega}_x\in\G^{t,x}_r$,  
which shows $A^c\in\Lambda$. Given $\{A_i\}_{i\in\N}\subset\Lambda$, we have $\left(\bigcup_{i\in\N}A_i\right)^{t,\omega}_x=\bigcup_{i\in\N}\{\tilde{\omega}\in\Omega^t \mid \omega\otimes_t \tilde{\omega}\in A_i, \tilde{\omega}_t=x\}=\bigcup_{i\in\N}(A_i)^{t,\omega}_x\in\G_r^{t,x}$,
which shows $\bigcup_{i\in\N}A_i\in\Lambda$. Thus, we conclude $\Lambda$ is a $\sigma$-algebra of $\Omega$. For any $x\in\Q^d$ and $\lambda\in\Q_+$, the set of positive rationals, let $O_\lambda(x)$ denote the open ball in $\R^d$ centered at $x$ with radius $\lambda$. Note from \cite[p.307]{KS-book-91} that for each $s\in[0,T]$, $\G^s_r$ is countably generated by 
\begin{equation}\label{countably generated}
\mathcal{C}^s_r:=\bigg\{\bigcap_{i=1}^m(W^s_{t_i})^{-1}(O_{\lambda_i}(x_i))\ \bigg|\ m\in\N,\ t_i\in\Q,\ s\le t_1<\cdots <t_m\le r,\ x_i\in\Q^d,\ \lambda_i\in\Q_+\bigg\}.
\end{equation}
Given $C=\bigcap_{i=1}^m(W_{t_i})^{-1}(O_{\lambda_i}(x_i))$ in $\mathcal{C}_r=\mathcal{C}^0_r$, if $t_m\ge t$, set $k=\min\{i=1,\cdots,m \mid t_i\ge t\}$; otherwise, set $k=m+1$. Then, if $\omega_{t_i}\notin O_{\lambda_i}(x_i)$ for some $i=1,\cdots,k-1$, we have $C^{t,\omega}_x=\emptyset\in\G^{t,x}_r$; if $k=m+1$ and $w_{t_i}\in O_{\lambda_i}(x_i)$ $\forall i=1,\cdots,m$, we have $C^{t,\omega}_x=(\Omega^t)_x\in\G^{t,x}_r$; for all other cases, 
\begin{equation}\label{C^{s,w}}
C^{t,\omega}_x=\{W^t_t=x\}\cap\bigcap_{i=k}^m(W^t_{t_i})^{-1}\left(O_{\lambda_i}(x_i-\omega_t+x)\right)\in\G^{t,x}_r.
\end{equation}
Thus, $\mathcal{C}_r\subseteq\Lambda$, which implies $\G_r=\sigma(\mathcal{C}_r)\subseteq\Lambda$. Now, for any $A\in\G_r$, $A^{t,\omega}_x\in\G^{t,x}_r\subseteq\G^t_r$.

(ii) Observe from part (i) that $\tilde{\omega}\in A^{t,\omega}\Leftrightarrow\tilde{\omega}\in A^{t,\omega}_{\tilde{\omega}_t}\Leftrightarrow \tilde{\omega}-\tilde{\omega}_t\in A_0^{t,\omega}$ i.e. $\psi_t(\tilde{\omega})\in A_0^{t,\omega}\Leftrightarrow \tilde{\omega}\in\psi^{-1}_t(A^{t,\omega}_0)$. Thus, $A^{t,\omega}=\psi_t^{-1}(A^{t,\omega}_0)\in\psi_t^{-1}(\G^{t,0}_r)=\cH^t_r\subseteq\G^t_r$, thanks to part (i) and \eqref{H in g}. Then, using part (i) again, $\P^t(A^{t,\omega})=\P^t(A^{t,\omega}_0)=\P^{t,x}(A^{t,\omega}_0+x)=\P^{t,x}(A^{t,\omega}_x)=\P^{t,x}(A^{t,\omega}),\ \forall x\in\R^d$.

(iii) By part (ii) and the Borel measurability of $\phi_t:(\Omega,\G_r)\mapsto(\Omega^t,\G^t_r)$, we immediately have $\phi_t^{-1} A^{t,\omega}\in\phi_t^{-1}\cH^t_r\subseteq\G_r$. Now, by property (e'') in \cite[p.84]{KS-book-91} and part (ii), 
\[
\P[\phi_t^{-1}A^{t,\omega}\mid\G_{t+}](\omega')=\P^{t,\omega'_t}(A^{t,\omega})=\P^t(A^{t,\omega})\ \ \hbox{for}\ \P\hbox{-a.e.}\ \omega'\in\Omega, 
\]
which implies $\P[\phi_t^{-1}A^{t,\omega}]=\P^t(A^{t,\omega})$.

(iv) Set $\Lambda:=\{\tilde{A}\subseteq \Omega^t \mid \omega\otimes_t\tilde{A}_{\omega_t}\in\G_r\}$. Let $\mathcal{C}^t_r$ be given as in \eqref{countably generated}. For any $C=\bigcap_{i=1}^m(W^t_{t_i})^{-1}(O_{\lambda_i}(x_i))$ in $\mathcal{C}^t_r$, we deduce from the continuity of paths in $\Omega$ that
\begin{equation*}
\begin{split}
\omega\otimes_t C_{\omega_t} &= \{\omega'\in\Omega \mid \omega'_s=\omega_s\ \forall s\in\Q\cap[0,t)\ \hbox{and}\ \omega'_{t_i}\in O_{\lambda_i}(x_i)\ \hbox{for}\ i=1,\cdots,m\}\\
&= \bigg(\bigcap_{s\in\Q\cap[0,t)}(W_{s})^{-1}(\omega_s)\bigg) \cap \bigg(\bigcap_{i=1}^m(W_{t_i})^{-1}(O_{\lambda_i}(x_i))\bigg) \in\G_r.
\end{split}
\end{equation*}
Thus, we have $\mathcal{C}^t_r\subseteq\Lambda$. Given $\{\tilde{A}_i\}_{i\in\N}\subset\Lambda$, we have $\omega\otimes_t (\bigcup_{i\in\N}\tilde{A}_i)_{\omega_t}=\bigcup_{i\in\N}(\omega\otimes_t(\tilde{A_i})_{\omega_t})\in\G_r$, which shows $\bigcup_{i\in\N}\tilde{A}_i\in\Lambda$; this in particular implies $\Omega^t=\bigcup_{n\in\N}(W^t_r)^{-1}(O_n(0))\in\Lambda$. Given $\tilde{A}\in\Lambda$, we have $\omega\otimes_t (\tilde{A}^c)_{\omega_t}=(\omega\otimes_t (\Omega^t)_{\omega_t})\setminus(\omega\otimes_t \tilde{A}_{\omega_t})\in\G_r$, which shows $\tilde{A}^c\in\Lambda$. Hence, $\Lambda$ is a $\sigma$-algebra of $\Omega^t$, which implies $\G^t_r=\sigma(\mathcal{C}^t_r)\subseteq\Lambda$. Now, by part (i), we must have $\omega\otimes_t A^{t,\omega}_{\omega_t}\in\G_r$.

(v) Since $\xi^{-1}(\mathcal{E})\in\G_r$, 
$(\xi^{t,\omega})^{-1}(\mathcal{E})=\{\omega'\in\Omega\mid\xi(\omega\otimes_t\phi_t(\omega'))\in\mathcal{E}\}=\{\omega'\in\Omega\mid\omega\otimes_t\phi_t(\omega')\in\xi^{-1}(\mathcal{E})\}=\phi^{-1}_t(\xi^{-1}(\mathcal{E}))^{t,\omega}\in\phi^{-1}_t\cH^t_r\subseteq\G_r$, thanks to part (iii).
\end{proof}

In light of Theorem 1.3.4 and equation (1.3.15) in \cite{Stroock-Varadhan-book-06}, for any $\mathbb{G}$-stopping time $\tau$, there exists a family $\{Q^\omega_\tau\}_{\omega\in\Omega}$ of probability measures on $(\Omega,\G_T)$, called a regular conditional probability distribution (r.c.p.d.) of $\P$ given $\G_\tau$, such that 
\begin{itemize}
\item [(i)] for each $A\in\G_T$, the mapping $\omega\mapsto Q^{\omega}_\tau(A)$ is $\G_{\tau}$-measurable.
\item [(ii)] for each $A\in\G_T$, it holds for $\P$-a.e. $\omega\in\Omega$ that
$\P[A\mid\G_{\tau}](\omega)=Q^{\omega}_{\tau}(A)$.
\item [(iii)] for each $\omega\in\Omega$, $Q^{\omega}_\tau\left(\omega\otimes_\tau(\Omega^{\tau(\omega)})_{\omega_{\tau}}\right)=1$.
\end{itemize}
By property (iii) above and Lemma~\ref{lem:measurability} (iv), for any fixed $\omega\in\Omega$, we can define a probability measure $Q^{\tau,\omega}$ on $\left(\Omega^{\tau(\omega)},\G^{\tau(\omega)}_T\right)$ by 
\[
Q^{\tau,\omega}(\tilde{A}):=Q^\omega_{\tau}(\omega\otimes_\tau \tilde{A}_{\omega_\tau}),\ \forall \tilde{A}\in\G^{\tau(\omega)}_T.
\]     
Then, combining properties (ii) and (iii) above, we have: for $A\in\G_T$, it holds for $\P$-a.e. $\omega\in\Omega$ that
\begin{equation}\label{pre-rcpd}
\P[A\mid\G_{\tau}](\omega)=Q^{\omega}_{\tau}\left((\omega\otimes_\tau(\Omega^{\tau(\omega)})_{\omega_{\tau}})\cap A\right)=Q^{\omega}_{\tau}(\omega\otimes_\tau A^{\tau,\omega}_{\omega_{\tau}})=Q^{\tau,\omega}(A^{\tau,\omega}).
\end{equation}  
Note that the r.c.p.d. $\{Q^\omega_\tau\}_{\omega\in\Omega}$ is generally not unique. For each $(t,x)\in[0,T]\times\R^d$, observe that the shifted Wiener measure $\P^{t,x}$ can be characterized as the unique solution to the martingale problem for the operator $L := \frac{1}{2}\sum_{i,j=1}^d\frac{\partial^2}{\partial x_i\partial x_j}$ starting from time $t$ with initial value $x$ (see \cite[Remark 7.1.23]{Stroock-book-11} and \cite[Exercise 6.7.3]{Stroock-Varadhan-book-06}). Then, thanks to the strong Markov property of solutions to the martingale problem (see e.g. \cite[Theorem 6.2.2]{Stroock-Varadhan-book-06}), there exists a particular r.c.p.d. $\{Q^\omega_\tau\}_{\omega\in\Omega}$ such that $Q^{\tau,\omega}=\P^{\tau(\omega),\omega_{\tau(\omega)}}$. Now, by \eqref{pre-rcpd} and Lemma~\ref{lem:measurability} (ii), we have: for $A\in\G_T,$
\begin{equation}\label{rcpd}
\P[A\mid\G_{\tau}](\omega)=\P^{\tau(\omega),\omega_{\tau(\omega)}}(A^{\tau,\omega})=\P^{\tau(\omega)}(A^{\tau,\omega}),\ \P\hbox{-a.s.}
\end{equation}

So far, we have restricted ourselves to $\mathbb{G}$-stopping times. We say a random variable $\tau:\Omega\mapsto[0,\infty]$ is a $\mathbb{G}$-optional time if $\{\tau<t\}\in\G_t$ for all $t\in[0,T]$. In the following, we obtain a generalized version of \eqref{rcpd} for $\mathbb{G}$-optional times.

\begin{lem}\label{lem:strong Markov}
Fix a $\mathbb{G}$-optional time $\tau\le T$. For any $A\in\G_T$, 
\[
\P[A\mid \G_{\tau+}](\omega)=\P^{\tau(\omega)}(A^{\tau,\omega})\ \hbox{for}\ \P\hbox{-a.e.}\ \omega\in\Omega.
\]
\end{lem}

\begin{proof}
{\bf Step 1:} By \cite[Problem 1.2.24]{KS-book-91}, we can take a sequence $\{\tau_n\}_{n\in\N}$ of $\mathbb{G}$-stopping times such that $\tau_{n}(\omega)\downarrow\tau(\omega)$ for all $\omega\in\Omega$. Fix $A\in\G_T$. For each $n\in\N$, \eqref{rcpd} implies that for any $B\in\G_{\tau_n}$,
\begin{equation}\label{E[1A1B]}
\E_{\P}[1_A1_B]=\E_{\P}[\P^{\tau_n(\omega)}(A^{\tau_n,\omega})1_B].
\end{equation}
Then, for any $B\in\G_{\tau+}$, we must have \eqref{E[1A1B]} for all $n\in\N$, since $\G_{\tau+}=\bigcap_{n\in\N}\G_{\tau_n}$. Now, by taking the limit in $n$ and assuming that for each $\omega\in\Omega$ 
\begin{equation}\label{limP^tau_n}
\lim_{n\to\infty}\P^{\tau_n(\omega)}(A^{\tau_n,\omega})= \P^{\tau(\omega)}(A^{\tau,\omega}),
\end{equation}
we obtain from the dominated convergence theorem that $\E_{\P}[1_A1_B]=\E_{\P}[\P^{\tau(\omega)}(A^{\tau,\omega})1_B]$. Since $B\in\G_{\tau+}$ is arbitrary, we conclude $\P[A\mid \G_{\tau+}](\omega)=\P^{\tau(\omega)}(A^{\tau,\omega})$ for $\P$-a.e. $\omega\in\Omega$.

{\bf Step 2:} It remains to prove \eqref{limP^tau_n}. Fix $\omega\in\Omega$ and set $\Lambda:=\{A\subseteq\Omega\mid \eqref{limP^tau_n}\ \hbox{holds}\}$. Since $\Omega^{s,\omega}=\Omega^{s}$, $\forall s\in[0,T]$, \eqref{limP^tau_n} holds for $\Omega$ and thus $\Omega\in\Lambda$. Given $A\in\Lambda$, we have 
$\P^{\tau_n(\omega)}[(A^c)^{\tau_n,\omega}]=\P^{\tau_n(\omega)}[(A^{\tau_n,\omega})^c]=1-\P^{\tau_n(\omega)}(A^{\tau_n,\omega})
\to 1-\P^{\tau(\omega)}(A^{\tau,\omega})=\P^{\tau(\omega)}[(A^{\tau,\omega})^c]=\P^{\tau(\omega)}[(A^c)^{\tau,\omega}]$,
which shows $A^c\in\Lambda$. Given a sequence $\{A_i\}_{i\in\N}$ of disjoint sets in $\Lambda$, observe that $\{A_i^{s,\omega}\}_{i\in\N}$ is a sequence of disjoint sets in $\Omega^s$ for any $s\in[0,T]$. Then we have
$\P^{\tau_n(\omega)}[(\bigcup_{i\in\N}A_i)^{\tau_n,\omega}]=\P^{\tau_n(\omega)}[\bigcup_{i\in\N}A_i^{\tau_n,\omega}]=\sum_{i\in\N}\P^{\tau_n(\omega)}(A_i^{\tau_n,\omega})
\to \sum_{i\in\N}\P^{\tau(\omega)}(A_i^{\tau,\omega})=\P^{\tau(\omega)}[\bigcup_{i\in\N}A_i^{\tau,\omega}]=\P^{\tau(\omega)}[(\bigcup_{i\in\N}A_i)^{\tau,\omega}]$,
which shows $\bigcup_{i\in\N}A_i\in\Lambda$. Thus, we conclude that $\Lambda$ is a $\sigma$-algebra of $\Omega$. 

As mentioned in the proof of Lemma~\ref{lem:measurability} (i), $\G_T$ is countably generated by $\mathcal{C}_T=\mathcal{C}^0_T$ given in \eqref{countably generated}. Given $C=\bigcap_{i=1}^m(W_{t_i})^{-1}(O_{\lambda_i}(x_i))$ in $\mathcal{C}_T$, if $t_m\ge \tau(\omega)$ we set $k:=\min\{i=1,\cdots,m\mid t_i\ge \tau(\omega)\}$; otherwise, set $k:=m+1$. We see that: {\bf 1.} If $\omega_{t_i}\notin O_{\lambda_i}(x_i)$ for some $i=1,\cdots,k-1$, then $C^{s,\omega}=\emptyset$ $\forall s\in[\tau(\omega),T]$ and thus \eqref{limP^tau_n} holds for $C$. {\bf 2.} If $k=m+1$ and $\omega_{t_i}\in O(x_i)$ for all $i=1,\cdots, m$, we have $C^{s,\omega}=\Omega^s$ $\forall s\in[\tau(\omega),T]$ and thus \eqref{limP^tau_n} still holds for $C$. {\bf 3.} For all other cases, $C^{s,\omega}_{\omega_s}$ is of the form in \eqref{C^{s,w}} $\forall s\in[\tau(\omega),T]$. Let $B$ be a $d$-dimensional Brownian motion defined on any given filtered probability space $(E,\mathcal{I},\{\mathcal{I}_s\}_{s\ge 0},P)$. Then by Lemma~\ref{lem:measurability} (ii),
\begin{equation*}
\begin{split}
&\hspace{0.25in}\P^{\tau_n(\omega)}[C^{\tau_n,\omega}]=\P^{\tau_n(\omega),\omega_{\tau_n(\omega)}}[C^{\tau_n,\omega}_{\omega_{\tau}}]=P[B_{t_i-\tau_n(\omega)}\in O_{\lambda_i}(x_i-\omega_{\tau_n(\omega)}),i=k\cdots,m]\\
&\to P[B_{t_i-\tau(\omega)}\in O_{\lambda_i}(x_i-\omega_{\tau(\omega)}),i=k\cdots,m]=\P^{\tau(\omega),\omega_{\tau(\omega)}}[C^{\tau,\omega}_{\omega_{\tau}}]=\P^{\tau(\omega)}[C^{\tau,\omega}]. 
\end{split}
\end{equation*}
Hence, we conclude that $\mathcal{C}_T\subseteq\Lambda$ and therefore $\G_T=\sigma(\mathcal{C}_T)\subseteq\Lambda$.
\end{proof}

Now, we want to generalize Lemma~\ref{lem:measurability} to incorporate $\mathbb{F}$-stopping times.

\begin{lem}\label{lem:measurability 2}
Fix $\theta\in\T$. We have
\begin{itemize}
\item [(i)] For any $\overline{N}\in\overline{\cN}$, $\overline{N}^{\theta,\omega}\in\overline{\cN}^{\theta(\omega)}$ and $\phi_{\theta}^{-1}\overline{N}^{\theta,\omega}\in\overline{\cN}$ for $\overline{\P}$-a.e. $\omega\in\Omega$. 
\item [(ii)] For any $r\in[0,T]$ and $A\in\F_r$, it holds for $\overline{\P}\hbox{-a.e.}\  \omega\in\Omega$ that
\[
\hbox{if}\ \theta(\omega)\le r,\ \ \ \  A^{\theta,\omega}\in\cH^{\theta(\omega)}_r\cup\overline{\cN}^{\theta(\omega)}\subseteq\oG^{\theta(\omega)}_r\ \  \hbox{and}\ \ \phi_{\theta}^{-1}A^{\theta,\omega}\in\F^{\theta(\omega)}_r. 
\]
\item [(iii)] For any $r\in[0,T]$ and $\xi\in L^0(\Omega, \F_r)$, it holds for $\overline{\P}\hbox{-a.e.}\ \omega\in\Omega$ that
\[
\hbox{if}\ \theta(\omega)\le r,\ \ \xi^{\theta,\omega}\in L^0(\Omega,\F^{\theta(\omega)}_r).
\]

\end{itemize}
\end{lem}

\begin{proof}
(i) Take $N\in\cN$ such that $\overline{N}\subseteq N$. By \cite[Exercise 2.7.11]{KS-book-91}, there exists a $\mathbb{G}$-optional time $\tau$ such that $\overline{N}_1:=\{\theta\neq\tau\}\in\overline{\cN}$. By Lemma~\ref{lem:strong Markov}, there exists $\overline{N}_2\in\cN\subset\overline{\cN}$ such that 
$0=\P[N\mid \G_{\tau+}](\omega)=\P^{\tau(\omega)}(N^{\tau,\omega}),\ \hbox{for}\ \omega\in\Omega\setminus\overline{N}_2$.
Thus, for $\omega\in\Omega\setminus(\overline{N}_1\cup\overline{N}_2)$, we have $0=\P^{\tau(\omega)}(N^{\tau,\omega})=\P^{\theta(\omega)}(N^{\theta,\omega})$, i.e. $N^{\theta,\omega}\in\cN^{\theta(\omega)}$. Since $\overline{N}^{\theta,\omega}\subseteq N^{\theta,\omega}$, we have $\overline{N}^{\theta,\omega}\in\overline{\cN}^{\theta(\omega)}$ $\overline{\P}$-a.s. 

On the other hand, from  Lemma~\ref{lem:measurability} (iii), $\P(\phi_{\theta}^{-1}N^{\theta,\omega})=\P^{\theta(\omega)}(N^{\theta,\omega})=0$ for $\omega\in\Omega\setminus(\overline{N}_1\cup\overline{N}_2)$, which shows $\phi_{\theta}^{-1}N^{\theta,\omega}\in\cN$ $\overline{\P}$-a.s. Since $\phi_{\theta}^{-1}\overline{N}^{\theta,\omega}\subseteq\phi_{\theta}^{-1}N^{\theta,\omega}$, we conclude $\phi_{\theta}^{-1}\overline{N}^{\theta,\omega}\in\overline{\cN}$ $\overline{\P}$-a.s.  

(ii) By \cite[Problem 2.7.3]{KS-book-91}, there exist $\tilde{A}\in\G_r$ and $\overline{N}\in\overline{\cN}$ such that $A=\tilde{A}\cup\overline{N}$ and $\tilde{A}\cap\overline{N}=\emptyset$. From Lemma~\ref{lem:measurability} (ii), we know that for any $\omega\in\Omega$, if $\theta(\omega)\le r$ then $\tilde{A}^{\theta,\omega}\in\cH^{\theta(\omega)}_r\subseteq\G^{\theta(\omega)}_r$. Also, from part (i) we have $\overline{N}^{\theta,\omega}\in\overline{\cN}^{\theta(\omega)}$ $\overline{\P}$-a.s. We therefore conclude that for $\overline{\P}$-a.e. $\omega\in\Omega$, if $\theta(\omega)\le r$, then $A^{\theta,\omega}=\tilde{A}^{\theta,\omega}\cup\overline{N}^{\theta,\omega}\in \cH^{\theta(\omega)}_r\cup\overline{\cN}^{\theta(\omega)}\subseteq\oG^{\theta(\omega)}_r$. Then, thanks to part (i) and Definition~\ref{defn:F}, it holds $\overline{\P}$-a.s. that $\phi_{\theta}^{-1}A^{\theta,\omega}=\phi_{\theta}^{-1}\tilde{A}^{\theta,\omega}\cup\phi_{\theta}^{-1}\overline{N}^{\theta,\omega}\in\phi_\theta^{-1}\cH^{\theta(\omega)}_r\cup\overline{\cN}\subseteq\F^{\theta(\omega)}_r$ if $\theta(\omega)\le r$.

(iii) Let $\mathcal{E}$ be a Borel subset of $\R$. Since $\xi^{-1}(\mathcal{E})\in\F_r$, we see from part (ii) that, for $\overline{\P}$-a.e. $\omega\in\Omega$, 
$\left(\xi^{\theta,\omega}\right)^{-1}(\mathcal{E})=\{\omega'\in\Omega\mid\xi(\omega\otimes_{\theta}\phi_{\theta}(\omega'))\in\mathcal{E}\}=\{\omega'\in\Omega\mid \omega\otimes_{\theta}\phi_{\theta}(\omega')\in\xi^{-1}(\mathcal{E})\}=\phi^{-1}_{\theta}(\xi^{-1}(\mathcal{E}))^{\theta,\omega}\in\F^{\theta(\omega)}_r$ if $\theta(\omega)\le r$.
\end{proof}

Now, we generalize Lemma~\ref{lem:strong Markov} to incorporate $\mathbb{F}$-stopping times.

\begin{lem}\label{lem:strong Markov 1-2}
Fix $\theta\in\T$. For any $A\in\F_T$, $\overline{\P}[A\mid\F_\theta](\omega)=\overline{\P}^{\theta(\omega)}(A^{\theta,\omega})$, for $\overline{\P}\hbox{-a.e.}\ \omega\in\Omega$.
\end{lem}

\begin{proof}
Thanks again to \cite[Exercise 2.7.11]{KS-book-91}, we may take a $\mathbb{G}$-optional time $\tau$ such that $\overline{N}_1:=\{\theta\neq\tau\}\in\overline{\cN}$ and $\F_{\tau}=\F_\theta$. Moreover, we have $A=\tilde{A}\cup\overline{N}$ for some $\tilde{A}\in\G_T$ and $\overline{N}\in\overline{\cN}$ with $\tilde{A}\cap\overline{N}=\emptyset$, by using \cite[Exercise 2.7.3]{KS-book-91}. Then, in view of Lemma~\ref{lem:measurability} (ii), Lemma~\ref{lem:measurability 2} (i), and Lemma~\ref{lem:strong Markov}, we can take some $\overline{N}_2\in\overline{\cN}$ such that for $\omega\in\Omega\setminus(\overline{N}_1\cup\overline{N}_2)$,
\begin{equation}\label{cond. prob.}
\begin{split}
\overline{\P}^{\theta(\omega)}(A^{\theta,\omega})&=\overline{\P}^{\tau(\omega)}(A^{\tau,\omega})=\overline{\P}^{\tau(\omega)}(\tilde{A}^{\tau,\omega})+\overline{\P}^{\tau(\omega)}(\overline{N}^{\tau,\omega})=\P^{\tau(\omega)}(\tilde{A}^{\tau,\omega})\\
&=\P[\tilde{A}\mid \G_{\tau+}](\omega)=\overline{\P}[\tilde{A}\mid \G_{\tau+}](\omega)=\overline{\P}[A\mid \G_{\tau+}](\omega).
\end{split}
\end{equation} 
For any $B\in\F_{\tau}$, $B=\tilde{B}\cup\overline{N}'$ for some $\tilde{B}\in\G_{\tau}\subseteq\G_{\tau+}$ and $\overline{N}'\in\overline{\cN}$ with $\tilde{B}\cap\overline{N}'=\emptyset$, thanks again to \cite[Exercise 2.7.3]{KS-book-91}. We then deduce from \eqref{cond. prob.} that $\E[1_{\tilde{A}}1_B]= \E[1_{\tilde{A}}1_{\tilde{B}}]=\E\left[\overline{\P}^{\theta(\omega)}(A^{\theta,\omega})1_{\tilde{B}}\right]=\E\left[\overline{\P}^{\theta(\omega)}(A^{\theta,\omega})1_B\right].$
Hence, we conclude $\overline{\P}^{\theta(\omega)}(A^{\theta,\omega})=\overline{\P}[A\mid\F_{\tau}](\omega)=\overline{\P}[A\mid\F_{\theta}](\omega)$, for $\omega\in\Omega\setminus(\overline{N}_1\cup\overline{N}_2)$. 
\end{proof}

Finally, we are able to generalize Lemma~\ref{lem:measurability} (iii) to incorporate $\mathbb{F}$-stopping times.

\begin{prop}\label{prop:strong Markov 2}
Fix $\theta\in\T$. We have
\begin{itemize}
\item [(i)] 
for any $A\in\F_T$, 
$\overline{\P}[A\mid\F_\theta](\omega)=\overline{\P}[\phi^{-1}_{\theta}A^{\theta,\omega}]$, for $\overline{\P}\hbox{-a.e.}\ \omega\in\Omega$.
\item [(ii)] for any $\xi\in L^1(\Omega,\F_T,\overline{\P})$, $\E[\xi\mid\F_\theta](\omega)=\E\left[\xi^{\theta,\omega}\right]$ for $\overline{\P}$-a.e. $\omega\in\Omega$.
\end{itemize}
\end{prop}

\begin{proof}
(i) By Lemma~\ref{lem:measurability 2} (i) and Lemma~\ref{lem:measurability} (iii), it holds $\overline{\P}$-a.s. that 
\[
\overline{\P}[\phi_{\theta}^{-1}A^{\theta,\omega}]=\overline{\P}[\phi_{\theta}^{-1}\tilde{A}^{\theta,\omega}]+\overline{\P}[\phi_{\theta}^{-1}\overline{N}^{\theta,\omega}]=\P[\phi_{\theta}^{-1}\tilde{A}^{\theta,\omega}]=\P^{\theta(\omega)}[\tilde{A}^{\theta,\omega}]=\overline{\P}^{\theta(\omega)}[\tilde{A}^{\theta,\omega}]=\overline{\P}^{\theta(\omega)}[A^{\theta,\omega}].
\] 
The desired result then follows from the above equality and Lemma~\ref{lem:strong Markov 1-2}.
 
(ii) Given $A\in\F_T$, observe that for any fixed $\omega\in\Omega$,  
$(1_A)^{\theta,\omega}(\omega')=1_A\left(\omega\otimes_{\theta}\phi_{\theta}(\omega')\right)
=1_{\phi_{\theta}^{-1}A^{\theta,\omega}}(\omega')$. Then we see immediately from part (i) that part (ii) is true for $\xi=1_A$. It follows that part (ii) also holds true for any $\F_T$-measurable simple function $\xi$. For any positive $\xi\in L^1(\Omega,\F_T,\overline{\P})$, we can take a sequence $\{\xi_n\}_{n\in\N}$ of $\F_T$-measurable simple functions such that $\xi_n(\omega)\uparrow\xi(\omega)\ \forall \omega\in\Omega$. By the monotone convergence theorem, there exists $\overline{N}\in\overline{\cN}$ such that $\E[\xi_n\mid\F_\theta](\omega)\uparrow\E[\xi\mid\F_\theta](\omega)$, for $\omega\in\Omega\setminus\overline{N}$. For each $n\in\N$, since $\xi_n$ is an $\F_T$-measurable simple function, there exists $\overline{N}_n\in\overline{\cN}$ such that $\E[\xi_n\mid\F_\theta](\omega)=\E\left[(\xi_n)^{\theta,\omega}\right]$, for $\omega\in\Omega\setminus\overline{N}_n$. Finally, noting that there exists $\overline{N}'\in\overline{\cN}$ such that $\xi^{\theta,\omega}$ is $\F_T$-measurable for $\omega\in\Omega\setminus\overline{N}'$ (from Lemma~\ref{lem:measurability 2} (iii)) and that
$(\xi_n)^{\theta,\omega}(\omega')\uparrow\xi^{\theta,\omega}(\omega')\ \forall \omega'\in\Omega$ (from the everywhere convergence $\xi_n\uparrow\xi$), we obtain from the monotone convergence theorem again that for $\omega\in\Omega\setminus\left((\bigcup_{n\in\N}\overline{N}_n)\cup\overline{N}\cup\overline{N}'\right)$, 
\[
\E[\xi\mid\F_\theta](\omega)=\lim_{n\to\infty}\E[\xi_n\mid\F_\theta](\omega)=\lim_{n\to\infty}\E[(\xi_n)^{\theta,\omega}]=\E[\xi^{\theta,\omega}].
\]
The same result holds true for any general $\xi\in L^1(\Omega,\F_T,\overline{\P})$ as $\xi=\xi^+-\xi^-$.
\end{proof}

\subsection{Proof of Proposition~\ref{prop:xi indep. of F_t}}\label{subsec:xi indep. of F_t}
\begin{proof}
(i) Set $\Lambda:=\{A\subseteq\Omega \mid \overline{\P}(A\cap B)=\overline{\P}(A)\overline{\P}(B)\ \forall B\in\F_t\}$. It can be checked that $\Lambda$ is a $\sigma$-algebra of $\Omega$. Take $A\in\phi^{-1}_t\cH^t_T\cup\overline{\cN}$. If $A\in\overline{\cN}$, it is trivial that $A\in\Lambda$; if $A=\phi^{-1}_t C$ with $C\in\cH^t_T$, then for any $B\in\F_t$,
\begin{equation*}
\overline{\P}(A\cap B)=\overline{\P}(B\cap \phi^{-1}_t C) = \E\left[\overline{\P}(B\cap\phi^{-1}_t C\mid\F_t)\right] = \E\left[\overline{\P}(B\cap\phi^{-1}_t C\mid\F_t)(\omega)1_B(\omega)\right].
\end{equation*}
By Proposition~\ref{prop:strong Markov 2} (i), for $\overline{\P}$-a.e. $\omega\in\Omega$, $\overline{\P}(B\cap\phi^{-1}_t C\mid\F_t)(\omega)=\overline{\P}[\phi^{-1}_t(B\cap\phi^{-1}_tC)^{t,\omega}]=\overline{\P}[\phi^{-1}_tC]=\overline{\P}(A)$ if $\omega\in B$. We therefore have $\overline{\P}(A\cap B)=\overline{\P}(A)\overline{\P}(B)$, and conclude $A\in\Lambda$. It follows that $\phi^{-1}_t\cH^t_T\cup\overline{\cN}\subseteq\Lambda$, which implies $\F^t_T=\sigma(\phi^{-1}_t\cH^t_T\cup\overline{\cN})\subseteq\Lambda$. Thus, $\F^t_T$ and $\F_t$ are independent. 

(ii) Let $\Delta$ denote the set operation of symmetric difference. Set $\Lambda:=\{A\subseteq\Omega\mid (\phi_t^{-1}A^{t,\omega})\Delta A\in\overline{\cN}\ \hbox{for}\ \overline{\P}\hbox{-a.e.}\ \omega\in\Omega\}$. It can be checked that $\Lambda$ is a $\sigma$-algebra of $\Omega$. Take $A\in\phi^{-1}_t\cH^t_T\cup\overline{\cN}$. If $A\in\overline{\cN}$, we see from Lemma~\ref{lem:measurability 2} (i) that $A\in\Lambda$; if $A=\phi^{-1}_tC$ with $C\in\cH^t_T$, then $\phi_t^{-1}A^{t,\omega}=\phi_t^{-1}C=A$ for all $\omega\in\Omega$, and thus $A\in\Lambda$. We then conclude that $\F^t_T=\sigma(\phi^{-1}_t\cH^t_T\cup\overline{\cN})\subseteq\Lambda$. 

Take a sequence $\{\xi_n\}$ of random variables in $L^0(\Omega,\F^t_T)$ taking countably many values $\{r_i\}_{i\in\N}$ such that $\xi_n(\omega)\to\xi(\omega)$ for all $\omega\in\Omega$. This everywhere convergence implies that for any fixed $\omega\in\Omega$, $(\xi_n)^{t,\omega}(\omega')\to\xi^{t,\omega}(\omega')$ for all $\omega'\in\Omega$. Now, fix $n\in\N$. For each $i\in\N$, since $(\xi_n)^{-1}\{r_i\}\in\F^t_T\subseteq\Lambda$, there exists $\overline{N}^n_i\in\overline{\cN}$ such that for $\omega\in\Omega\setminus\overline{N}^n_i$,
\begin{equation}\label{symmetric difference}
\left(\left[(\xi_n)^{t,\omega}\right]^{-1}\{r_i\}\right)\Delta(\xi_n)^{-1}\{r_i\}=\left[\phi^{-1}_t\left((\xi_n)^{-1}\{r_i\}\right)^{t,\omega}\right]\Delta(\xi_n)^{-1}\{r_i\}=:\overline{M}^n_i\in\overline{\cN},
\end{equation}
where the first equality follows from the calculation in the proof of Lemma~\ref{lem:measurability 2} (iii). Then, we deduce from \eqref{symmetric difference} that: for any fixed $\omega\in\Omega\setminus\bigcup_{i\in\N}\overline{N}^n_i$, $(\xi_n)^{t,\omega}(\omega')=\xi_n(\omega')$ for all $\omega'\in\Omega\setminus\bigcup_{i\in\N}\overline{M}^n_i$. It follows that: for any fixed $\omega\in\Omega\setminus\bigcup_{i,n\in\N}\overline{N}^n_i$, $(\xi_n)^{t,\omega}(\omega')=\xi_n(\omega')$ for all $\omega'\in\Omega\setminus\bigcup_{i,n\in\N}\overline{M}^n_i$ and $n\in\N$. Setting $\overline{N}=\bigcup_{i,n\in\N}\overline{N}^n_i$ and $\overline{M}=\bigcup_{i,n\in\N}\overline{M}^n_i$, we obtain that for any $\omega\in\Omega\setminus\overline{N}$,
\[
\xi(\omega')=\lim_{n\to\infty}\xi_n(\omega')=\lim_{n\to\infty}(\xi_n)^{t,\omega}(\omega')=\xi^{t,\omega}(\omega'),\ \hbox{for}\ \omega'\in\Omega\setminus\overline{M}.
\]
\end{proof}

\subsection{Proof of Proposition~\ref{lem: tau a stopping time a.s.}}\label{subsec: tau a stopping time a.s.}
\begin{proof}
Take a sequence of stopping times $\{\tau_i\}_{i\in\N}\subset \T$ such that $\tau_i$ takes values in $\{m/2^i\mid m\in\N\}$ for each $i\in\N$ and $\tau_i(\omega)\downarrow\tau(\omega)$ for all $\omega\in\Omega$ (thanks to \cite[Problem 1.2.24]{KS-book-91}). Set $\overline{N}:=\{\tau<\theta\}\in\overline{\cN}$. Since $\tau_i(\omega)\downarrow\tau(\omega)$ for all $\omega\in\Omega$, we have $\tau_i\ge\theta$ on $\Omega\setminus\overline{N}$ for all $i\in\N$. For each $i\in\N$, let $r^i_m:=m/2^i$, $m\in\N$. Since $\{\tau_i\le r^i_m\}\in\F_{r^i_m}$ for all $m\in\N$, we deduce from Lemma~\ref{lem:measurability 2} (ii) and the countability of $\{r^i_m\}_{m\in\N}$ that there exists $\overline{N}^i\in\overline{\cN}$ such that for $\omega\in\Omega\setminus \overline{N}^i$,
\begin{equation}\label{{tau_i<r}}
\hbox{if}\ \theta(\omega)\le r^i_m,\ \phi_\theta^{-1}\{\tau_i\le r^i_m\}^{\theta,\omega}\in\F^{\theta(\omega)}_{r^i_m}\ \hbox{for all}\ m\in\N.
\end{equation}
Fix $r\in[0,T]$. For any $\omega\in\Omega\setminus(\overline{N}\cup\overline{N}^i)$, if $\theta(\omega)> r$, then $\tau_i(\omega)\ge\theta(\omega)>r$ and thus $\phi_\theta^{-1}\{\tau_i\le r\}^{\theta,\omega}=\phi^{-1}_\theta\emptyset=\emptyset\in\F^{\theta(\omega)}_r$; if $\theta(\omega)\le r$, there are two cases: {\bf 1.} $\exists$ $m^*\in\N$ s.t. $r^i_{m^*}\in[\theta(\omega),r]$ and $r^i_{m^*+1}>r$. Then, by \eqref{{tau_i<r}}, $\phi^{-1}_\theta\{{\tau_i\le r}\}^{\theta,\omega}=\phi^{-1}_\theta\{\tau_i\le r^i_{m^*}\}^{\theta,\omega}\in\F^{\theta(\omega)}_{r^i_{m^*}}\subset\F^{\theta(\omega)}_r$; {\bf 2.} $\exists$ $m^*\in\N$ s.t. $r^i_{m^*}<\theta(\omega)$ and $r^i_{m^*+1}>r$. Since $\tau_i(\omega)\ge\theta(\omega)>r^i_{m^*}$, $\phi_\theta^{-1}\{\tau_i\le r\}^{\theta,\omega}=\phi^{-1}_\theta\{\tau_i\le r^i_{m^*}\}^{\theta,\omega}=\phi^{-1}_\theta\emptyset=\emptyset\in\F^{\theta(\omega)}_r$. Thus, for $\omega\in\Omega\setminus(\overline{N}\cup\overline{N}^i)$, we have $\phi_\theta^{-1}\{\tau_i\le r\}^{\theta,\omega}\in\F^{\theta(\omega)}_r$, and therefore 
\begin{equation*}
\{\tau_i^{\theta,\omega}\le r\} = \{\tau_i\left(\omega\otimes_\theta\phi_{\theta}(\omega')\right)\le r\}
=\phi_{\theta}^{-1}\{\tau_i\le r\}^{\theta,\omega}\in\F^{\theta(\omega)}_r,\ \forall\ r\in[0,T].
\end{equation*}
This shows that $\tau_i^{\theta,\omega}\in\T^{\theta(\omega)}_{\theta(\omega),T}$ for $\omega\in\Omega\setminus(\overline{N}\cup\overline{N}^i)$. Hence, for $\omega\in\Omega\setminus\left(\overline{N}\cup(\bigcup_{i\in\N}\overline{N}^i)\right)$, we have $\tau_i^{\theta,\omega}\in\T^{\theta(\omega)}_{\theta(\omega),T}$ $\forall i\in\N$. Finally, since the filtration $\mathbb{F}^{\theta(\omega)}$ is right-continuous,  $\tau^{\theta,\omega}(\omega')=\downarrow\lim_{i\to\infty}\tau_i^{\theta,\omega}(\omega')$ (this is true since $\tau_i\downarrow\tau$ everywhere) must also be a stopping time in $\T^{\theta(\omega)}_{\theta(\omega),T}$.
\end{proof}

\subsection{Proof of Proposition~\ref{prop:alpha^t,w=alpha}}\label{subsec:alpha^t,w=alpha}
Recall the metric $\tilde{\rho}$ on $\A$ defined in \eqref{metric on A}. We say $\beta\in\A$ is a step control if there exists a subdivision $0=t_0<t_1<\cdots<t_m=T$, $m\in\N$, of the interval $[0,T]$ such that $\beta_t=\beta_{t_i}$ for $t\in[t_i,t_{i+1})$ for $i=0,1,\cdots,m-1$.

\begin{proof}
By \cite[Lemma 3.2.6]{Krylov-book-80}, there exist a sequence $\{\alpha^n\}$ of step controls such that $\alpha^n\to\alpha$. For each $n\in\N$, in view of Proposition~\ref{prop:xi indep. of F_t} (ii), there exist $\overline{N}_n,\overline{M}_n\in\overline{\cN}$ such that: for any fixed $\omega\in\Omega\setminus\overline{N}_n$, $(\alpha^n_r)^{t,\omega}(\omega')=\alpha^n_r(\omega')$ for $(r,\omega')\in[0,T]\times(\Omega\setminus\overline{M}_n)$. It follows that: for any fixed $\omega\in\Omega\setminus\bigcup_{n\in\N}\overline{N}_n$, $(\alpha^n_r)^{t,\omega}(\omega')=\alpha^n_r(\omega')$ for all $(r,\omega')\in[0,T]\times(\Omega\setminus\bigcup_{n\in\N}\overline{M}_n)$ and $n\in\N$. With the aid of Proposition~\ref{prop:strong Markov 2} (ii), we obtain
\begin{equation*}
\begin{split}
0&=\lim_{n\to\infty}\tilde{\rho}(\alpha^n,\alpha)=\lim_{n\to\infty}\E\left[\int_0^T\rho'(\alpha^n_r,\alpha_r)dr\right]=\lim_{n\to\infty}\E\left(\E\left[\int_0^T\rho'(\alpha^n_r,\alpha_r)dr\ \middle|\ \F_t\right](\omega)\right)\\
&=\lim_{n\to\infty}\int\int\bigg(\int_0^T\rho'(\alpha^n_r,\alpha_r)dr\bigg)^{t,\omega}(\omega')\ d\overline{\P}(\omega')\ d\overline{\P}(\omega)\\
&=\lim_{n\to\infty}\int\int\int_0^T\rho'\bigg((\alpha^n_r)^{t,\omega}(\omega'),\alpha^{t,\omega}_r(\omega')\bigg)dr\ d\overline{\P}(\omega')\ d\overline{\P}(\omega)\\
&=\lim_{n\to\infty}\int \tilde{\rho}\left((\alpha^n)^{t,\omega},\alpha^{t,\omega}\right) d\overline{\P}(\omega)=\lim_{n\to\infty}\int \tilde{\rho}(\alpha^n,\alpha^{t,\omega}) d\overline{\P}(\omega)=\int\lim_{n\to\infty}\tilde{\rho}(\alpha^n,\alpha^{t,\omega}) d\overline{\P}(\omega),
\end{split}
\end{equation*}
where the last equality is due to the dominated convergence theorem. This implies that 
$0=\lim_{n\to\infty}\tilde{\rho}(\alpha^n,\alpha^{t,\omega}),\ \hbox{for}\ \overline{\P}\hbox{-a.e.}\ \omega\in\Omega$. Recalling that $\alpha^n\to\alpha$, we conclude that $\tilde{\rho}(\alpha^{t,\omega},\alpha)=0$ for $\overline{\P}$-a.e. $\omega\in\Omega$. The second assertion follows immediately from \cite[Exercise 3.2.4]{Krylov-book-80}.
\end{proof}


\subsection{Proof of Lemma~\ref{lem:E = J}}\label{subsec:E = J}
\begin{proof}
By taking $\xi=F(\bX^{t,\bx,\alpha}_\tau)$ in Proposition~\ref{prop:strong Markov 2} (ii) and using Remark~\ref{rem:flow property} (ii),
\begin{equation*}
\begin{split}
&\hspace{0.2in}\E[F(\bX^{t,\bx,\alpha}_\tau)\mid\F_\theta](\omega)=\E\left[F(\bX^{t,\bx,\alpha}_\tau)^{\theta,\omega}\right]=\int F\left(\bX^{t,\bx,\alpha}_\tau\left(\omega\otimes_\theta\phi_\theta(\omega')\right)\right)d\overline{\P}(\omega')\\
&=\int F\left(\bX^{\theta(\omega),\bX^{t,\bx,\alpha}_{\theta}(\omega),\alpha^{\theta,\omega}}_{\tau^{\theta,\omega}}(\omega')\right)d\overline{\P}(\omega')
=J\left(\theta(\omega),\bX^{t,\bx,\alpha}_{\theta}(\omega);\alpha^{\theta,\omega},\tau^{\theta,\omega}\right),\ \hbox{for}\ \overline{\P}\hbox{-a.e.}\ \omega\in\Omega.
\end{split}
\end{equation*} 
\end{proof}

\bibliographystyle{siam}
\bibliography{refs}

\end{document}